\documentclass[final]{siamart0516}

\usepackage{amsmath,amssymb}
\usepackage{url}
\usepackage{hyperref}
\usepackage[ruled,vlined, linesnumbered, algo2e]{algorithm2e}
\usepackage{algorithmicx,algpseudocode,algcompatible}
\usepackage{mathtools}
\usepackage{multirow}
\usepackage{graphicx}

\newsiamremark{remark}{Remark}
\newtheorem{assumption}{Assumption}[section]

\newcommand{\set}[1]{\left\{#1\right\}}
\newcommand{\sets}[1]{\{#1\}}
\newcommand{\norm}[1]{\left\Vert#1\right\Vert}
\newcommand{\norms}[1]{\Vert#1\Vert}
\newcommand{\prox}{\mathrm{prox}}

\newcommand{\dom}[1]{\mathrm{dom}(#1)}
\newcommand{\iprods}[1]{\langle #1\rangle}
\newcommand{\ri}[1]{\mathrm{ri}\left(#1\right)}
\newcommand{\argmin}{\mathrm{arg}\!\displaystyle\min}

\newcommand{\R}{\mathbb{R}}
\newcommand{\Rext}{\R\cup\{+\infty\}}

\newcommand{\Xc}{\mathcal{X}}
\newcommand{\Yc}{\mathcal{Y}}
\newcommand{\Sc}{\mathcal{S}}
\newcommand{\Dc}{\mathcal{D}}

\newcommand{\Lc}{\mathcal{L}}

\newcommand{\Tc}{\mathcal{T}}

\newcommand{\Gc}{\mathcal{G}}
\newcommand{\Nc}{\mathcal{N}}

\newcommand{\BigO}[1]{\mathcal{O}\left(#1\right)}
\newcommand{\SmallO}[1]{o\left(#1\right)}
\newcommand{\SmallOs}[1]{o\big(#1\big)}

\newcommand{\LowerO}[1]{\Omega\left(#1\right)}

\newcommand{\mytbi}[1]{\textbf{\textit{#1}}}

\newcommand{\beforesec}{\vspace{0ex}}
\newcommand{\aftersec}{\vspace{0ex}}
\newcommand{\beforesubsec}{\vspace{0ex}}
\newcommand{\aftersubsec}{\vspace{0ex}}

\newcommand{\beforepara}{\vspace{0ex}}

\newcommand{\myeq}[2]{\vspace{-0.25ex}\begin{equation}\label{#1}#2\vspace{-0.5ex}\end{equation}}
\newcommand{\myeqn}[1]{\vspace{-0.25ex}\begin{equation*}#1\vspace{-0.5ex}\end{equation*}}
\newcommand{\myeqf}[2]{\vspace{-0.5ex}\begin{equation}\label{#1}#2\vspace{-0.5ex}\end{equation}}
\newcommand{\myeqfn}[1]{\vspace{-0.5ex}\begin{equation*}#1\vspace{-0.5ex}\end{equation*}}

\headers{Non-Stationary First-Order Primal-Dual Algorithms}{Q. Tran-Dinh ~\text{\lowercase{and}}~ Y. Zhu}

\title{Non-Stationary First-Order Primal-Dual Algorithms with Faster Convergence Rates}

\author{
Quoc Tran-Dinh$^{\ast}$~~\and~~Yuzixuan Zhu\thanks{Department of Statistics and Operations Research, 
The University of North Carolina at Chapel Hill (UNC), 318-Hanes Hall, Chapel Hill, NC, 27599-3260, USA. 
\newline Corresponding author: {\tt quoctd@email.unc.edu}.} 
}

\ifpdf
\hypersetup{
  pdftitle={Non-Stationary First-Order Primal-Dual Algorithms with Faster Convergence Rates},
  pdfauthor={Quoc Tran-Dinh ~\textit{and}~ Yuzixuan Zhu}
}
\fi

\begin{document}
\maketitle

\begin{abstract}
In this paper, we propose two novel non-stationary first-order primal-dual algorithms to solve nonsmooth composite convex optimization problems.
Unlike existing primal-dual schemes where the parameters are often fixed, our methods use pre-defined and dynamic sequences for parameters.
We prove that our first algorithm can achieve $\BigO{1/k}$ convergence rate on the primal-dual gap, and primal and dual objective residuals, where $k$ is the iteration counter. 
Our rate is on the non-ergodic (i.e., the last iterate) sequence of the primal problem and on the ergodic (i.e., the averaging) sequence of the dual problem, which we call semi-ergodic rate. 
By modifying the step-size update rule, this rate can be boosted even faster on the primal objective residual.
When the problem is strongly convex, we develop a second primal-dual algorithm that exhibits $\BigO{1/k^2}$ convergence rate on the same three types of guarantees.
Again by modifying the step-size update rule, this rate becomes faster on the primal objective residual.
Our primal-dual algorithms are the first ones to achieve such fast convergence rate guarantees under mild assumptions compared to existing works, to the best of our knowledge.
As byproducts, we apply our algorithms to solve constrained convex optimization problems and prove the same convergence rates on both the objective residuals and the feasibility violation.
We  still obtain at least $\BigO{1/k^2}$ rates even when the problem is ``semi-strongly'' convex.
We verify our theoretical results via two well-known numerical examples.

\vspace{1ex}
\noindent\textbf{Keywords:}
Non-stationary primal-dual method;  non-ergodic convergence rate; fast convergence rates; composite convex minimization; constrained convex optimization.
\end{abstract}

\noindent
\begin{AMS}
90C25, 90C06, 90-08
\end{AMS}

\beforesec
\section{Introduction}\label{sec:intro}
\aftersec
\beforepara
\paragraph{Problem statement}
In this paper, we develop new  first-order  primal-dual algorithms to solve the following classical composite convex minimization problem:
\myeq{eq:com_cvx}{
F^{\star} := \min_{x\in\R^p}\Big\{ F(x) := f(x) + g(Kx) \Big\},
}
where $f: \R^p\to\Rext$ and $g: \R^n\to\Rext$ are two proper, closed, and convex functions, and $K: \R^p\to\R^n$ is a given general linear operator.
Associated with the primal problem \eqref{eq:com_cvx}, we also consider its dual form as
\myeq{eq:dual_prob}{
G^{\star} := \min_{y\in\R^n}\Big\{ G(y) := f^{\ast}(-K^{\top}y) + g^{\ast}(y) \Big\},
}
where $f^{\ast}$ and $g^{\ast}$ are the Fenchel conjugates of $f$ and $g$, respectively.
We can combine the primal and dual problems \eqref{eq:com_cvx} and \eqref{eq:dual_prob} into the following min-max setting:
\myeq{eq:min_max}{
\min_{x\in\R^p}\max_{y\in\R^n}\Big\{ \widetilde{\Lc}(x, y) := f(x) + \iprods{Kx, y} - g^{\ast}(y) \Big\},
}
where $\widetilde{\Lc}(x, y)$ can be referred to as the Lagrange function of \eqref{eq:com_cvx} and \eqref{eq:dual_prob}, see \cite{Bauschke2011}.

\beforepara
\paragraph{A brief overview of primal-dual methods}
The study of first-order primal-dual methods for solving \eqref{eq:com_cvx} and \eqref{eq:dual_prob} has become extremely active in recent years, ranging from algorithmic development and convergence theory to applications, see, e.g., \cite{Bauschke2011,chambolle2016introduction,Esser2010a,glowinski2017splitting}.
This type of methods has close connection to other fields such as monotone inclusions, variational inequalities, and game theory \cite{Bauschke2011,Facchinei2003}.
They also have various direct applications in image and signal processing, machine learning, statistics, economics, and engineering, see, e.g., \cite{Chambolle2011,combettes2012primal,Esser2010}.

In our view, the study of first-order primal-dual methods for convex optimization can be divided into three main streams.
The first one is algorithmic development with numerous variants using different frameworks such as fixed-point theory, projective methods, monotone operator splitting, Fenchel duality and augmented Lagrangian frameworks, and variational inequality, see, e.g., \cite{BricenoArias2011,chen2016primal,Chen2013a,combettes2012primal,Esser2010,Goldstein2012,He2012b,malitsky2016first,Monteiro2011,Monteiro2010,spingarn1983partial,xu2017accelerated,yan2018new,Zhang2011,Zhu2008}.
Among different primal-dual variants for convex optimization, the general primal-dual hybrid gradient (PDHG) method proposed in \cite{Chambolle2011,Esser2010,pock2009algorithm,Zhu2008} appears to be the most general scheme that covers many existing variants, as investigated in \cite{chambolle2016introduction,Goldstein2015adaptive,conner2017equivalence}.
Using an appropriate reformulation of \eqref{eq:com_cvx}, \cite{conner2017equivalence} showed that the general PDHG scheme is in fact equivalent to Douglas-Rachford's splitting method \cite{Bauschke2011,Eckstein1992,Lions1979}, and, therefore, to ADMM in the dual setting.
Extensions to three operators and three  objective functions have also been studied in several works, including \cite{boct2015convergence,Condat2013,Davis2015,vu2013splitting}.
Other extensions to non-bilnear terms, Bregman distances, multi-objective terms, and stochastic variants have been also intensively investigated, see, e.g., in \cite{Bauschke2011,chambolle2017stochastic,Nemirovskii2004,hien2017inexact,tseng1990further,hamedani2018primal,yin2008bregman}.

The second stream is convergence analysis.
Existing works often use a gap function to measure the optimality of given approximate solutions \cite{Chambolle2011,Nemirovskii2004}.
This approach usually combines both primal and dual variables in one and uses, e.g.,  variational inequality frameworks to prove convergence, see, e.g., \cite{Eckstein1992,He2012b,Monteiro2011,Monteiro2010}.
An algorithmic-independent framework to characterize primal-dual gap certificates can be found in \cite{dunner2016primal}. 
Together with asymptotic convergence and linear convergence rates, many researchers have recently focused on sublinear convergence rates under weaker assumptions than strong convexity and smoothness or strong monotonicity-type and Lipschitz continuity conditions, see \cite{Bot2012,boct2015convergence,chambolle2016ergodic,Chen2013a,Davis2014,Davis2014b,He2012b,he2016accelerated,Monteiro2011,Monteiro2010,TranDinh2015b,xu2017accelerated} for more details.  
We emphasize that in general convex settings, such convergence rates are often achieved via averaging sequences on both primal and dual variables, which are much faster and easier to derive than on the sequence of last iterates.

The third stream is applications, especially in image and signal processing, see, e.g., \cite{Chambolle2011,chambolle2016introduction,combettes2011,combettes2011proximal,Esser2010a,connor2014primal}.
Recently, many primal-dual methods have also been applied to solving problems from machine learning, statistics, and engineering, see, e.g., \cite{chambolle2016introduction,glowinski2017splitting}. 
While theoretical results have shown that primal-dual methods may suffer from slow sublinear convergence rates under mild assumptions, their empirical convergence rates  are much better on concrete applications \cite{Chambolle2011,Esser2010}.

\beforepara 
\paragraph{Motivation}
In many applications, the desired solutions often have special structures such as sharp-edgedness in images, sparsity in signal processing and model selection, and low-rankness in matrix approximation. 
Such structures can be modeled using regularizers, constraints, or penalty functions, but unfortunately can be destroyed by algorithms that use \mytbi{ergodic} (i.e., averaging or weighted averaging) sequences as outputs, which is one of the reasons why many algorithms eventually take the \mytbi{non-ergodic} (i.e., last iterate) sequence as output while ignoring the fact that their convergence rate guarantee is proved based on an ergodic sequence.
In addition,  as observed in \cite{Davis2015}, the last-iterate sequence often has fast empirical convergence rate (e.g., up to linear rate).
This mismatch between theory and practice motivates us to develop new primal-dual algorithms that return the last iterates as outputs with rigorous convergence rate guarantees.
While non-ergodic convergence guarantees have recently been discussed in \cite{Davis2014a,Davis2014b} for several methods, it did not achieve the \textit{optimal} rate.
This paper develops two new first-order primal-dual schemes to fulfill this gap by using dynamic step-sizes, which leads to non-stationary methods, where the term ``non-stationary'' is adopted from \cite{liang2017local} for Douglas-Rachford methods.

Whereas $\BigO{1/k}$ convergence rate appears to be optimal under only convexity and strong duality assumptions when $k \leq \mathcal{O}(p)$, faster convergence rate for $k > \mathcal{O}(p)$ in primal-dual methods seems to not be known yet, especially in non-ergodic sense.
Recently, \cite{attouch2016rate} showed that Nesterov's accelerated method can exhibit up to $\SmallO{1/k^2}$ convergence rate when $k$ is sufficiently large compared to the problem dimension $p$.
This rate can only be achieved if $g$ has Lipschitz continuous gradient. 
This motivates us to consider such an acceleration in first-order primal-dual methods by adopting the approach in \cite{attouch2016rate}.
We show $\underline{o}\big(1/(k\sqrt{\log k})\big)$ non-ergodic convergence rate on the objective residual sequence in the sense that $\liminf_{k \to \infty} (k\sqrt{\log k})[F(x^k) - F^{\star}] = 0$ (\emph{cf.} \eqref{eq:underline_o_def})  without any smoothness or strong convexity-type assumption. 
A similar type of rate is also proved in \cite{Davis2014,Davis2014b} with $\SmallOs{1/\sqrt{k}}$ rate under the same assumption as ours, and $\SmallO{1/k}$ rate under additional assumption of strong convexity or smoothness (our non-ergodic rates are both $\BigO{1/k^2}$ and $\underline{o}\big(1/(k^2\sqrt{\log{k}})\big)$ in this case).

\beforepara
\paragraph{Our contributions}
To this end, our contributions are summarized as follows:

\begin{itemize}
\item[(a)]
We develop a new first-order primal-dual scheme, Algorithm \ref{alg:A1}, to solve primal and dual problems \eqref{eq:com_cvx} and \eqref{eq:dual_prob}. 
We prove the $\BigO{1/k}$ optimal convergence rate on three criteria: primal-dual gap, primal objective residual, and dual objective residual under only convexity and strong duality assumptions. 
Our guarantee is achieved in semi-ergodic sense, i.e., non-ergodic in primal variable and ergodic in dual variable.
For sufficiently large $k$ (i.e., $k > \BigO{p}$), by modifying the parameter update rules, we can show that our algorithm can be boosted up to $\min\sets{\BigO{1/k},\ \underline{o}(1/(k\sqrt{\log{k}}))}$ non-ergodic convergence rate on the primal objective residual. 
This rate is not slower than $\BigO{1/k}$ and empirically significantly faster than its counterpart with only $\BigO{1/k}$ rate.

\item[(b)]
If we apply Algorithm~\ref{alg:A1} to solve nonsmooth constrained convex optimization problems, then we can prove the same $\BigO{1/k}$ and  $\underline{o}\big(1/(k\sqrt{\log{k}})\big)$ convergence rates on the primal objective residual and the feasibility violation.

\item[(c)]
If $f$ of \eqref{eq:com_cvx} is strongly convex (or equivalently, its Fenchel conjugate $f^{\ast}$ is $L$-smooth), then we propose another first-order primal-dual algorithm, Algorithm \ref{alg:A2}, which achieves the $\BigO{1/k^2}$ optimal convergence rate on the same three criteria as of Algorithm~\ref{alg:A1}.
When $k$ is sufficiently large (i.e., $k > \BigO{p}$), by modifying the parameter update rules of Algorithm~\ref{alg:A2}, we obtain up to $\min\sets{\BigO{1/k^2},\ \underline{o}(1/(k^2\sqrt{\log{k}}))}$ non-ergodic convergence rates on \eqref{eq:com_cvx}.

\item[(d)]
If we modify Algorithm \ref{alg:A2} to solve the constrained convex problem \eqref{eq:constr_cvx2}, where the objective is semi-strongly convex (i.e., one objective term is strongly convex while the other term is non-strongly convex), then we prove the same $\BigO{1/k^2}$ and $\underline{o}\big(1/(k^2\sqrt{\log{k}})\big)$ rates  for both the primal objective residual and the feasibility violation.
\end{itemize}

\beforepara
\paragraph{Comparison}
We highlight some key differences between our algorithms and existing methods in terms of approach, algorithmic appearance, and theoretical guarantees.
First, unlike existing augmented Lagrangian-based methods, we view the augmented term  as a smoothed term for the indicator of linear constraints in the constrained reformulation \eqref{eq:constr_reform} of \eqref{eq:com_cvx}.
Next, we apply Nesterov's accelerated scheme to minimize this smoothed Lagrange function and simultaneously update the smoothness parameter (i.e., the penalty parameter) at each iteration in a homotopy fashion.

Second, Algorithm~\ref{alg:A1} has similar structure as Chambolle-Pock's method \cite{Chambolle2011,chambolle2016ergodic,conner2017equivalence}, a special case of PDHG, but it possesses a three-point momentum step depending on the iterates at the iterations $k$, $k-1$, and $k-2$, and makes use of dynamic parameters and step-sizes.
Algorithm~\ref{alg:A2} uses two proximal operators of the primal objective to obtain a non-ergodic convergence rate (but not required, see Subsection~\ref{subsec:alg2_full}).

Third, unlike existing works where the best-known convergence rates are often obtained via ergodic sequences, see, e.g., \cite{chambolle2016ergodic,Davis2014a,Davis2014b,He2012b,he2016accelerated,Monteiro2011,Monteiro2010}, our methods achieve the optimal convergence rates  in non-ergodic sense.
The $\BigO{1/k}$ ergodic optimal rate of primal-dual methods for solving \eqref{eq:com_cvx} is not new and has been proved in many papers. 
Their non-ergodic rates have just recently been proved, e.g., in \cite{tran2017proximal,tran2018adaptive,TranDinh2015b,valkonen2019inertial}. 
More precisely, \cite{tran2018adaptive, TranDinh2015b} utilize the Nesterov's smoothing technique in \cite{Nesterov2005c} and only derive primal convergence rates.
\cite{tran2017proximal} only handles constrained problems by applying the quadratic penalty function approach.
 \cite{valkonen2019inertial} relies on the well-known Chambolle-Pock scheme in \cite{Chambolle2011} by adding  inertial correction terms and adapting the parameters to achieve non-ergodic rates.
Nevertheless, our algorithm in this paper uses a completely different approach and achieves the $\BigO{1/k}$ rate on three criteria.

Finally, in addition to the $\BigO{1/k}$ non-ergodic rate on the primal objective residual, we also prove its $\underline{o}\big(1/(k\sqrt{\log{k}})\big)$ non-ergodic rate.
In comparison, \cite{Davis2014a} provides an intensive analysis of convergence rates for several methods to solve a more general problem than \eqref{eq:com_cvx}.
However, \cite{Davis2014a} does not provide new algorithms, and its convergence rate if applied to \eqref{eq:com_cvx} would become $\SmallOs{1/\sqrt{k}}$. 
Other related works include \cite{Davis2014, davis2016convergence,Davis2014b,Davis2015}.
Table \ref{pdne:table:convergence_results} non-exhaustively summarizes the best-known convergence rates of first-order primal-dual methods for solving \eqref{eq:com_cvx}, where we highlight that this paper contributes the fastest rates under corresponding assumptions.

\begin{table}[htbp!]
	\centering
	\caption{State-of-the-art results and our contributions to convergence rates of the primal objective residual sequence $\set{F(x) - F^\star}$ of first-order primal-dual algorithms for solving \eqref{eq:com_cvx}.
	Here, $f$ and $g$ are convex and possibly nonsmooth, and $g$ is Lipschitz continuous. In addition, we consider the assumption where $f$ is strongly convex.}
	\vspace{-1.5ex}
	\label{pdne:table:convergence_results}
	\resizebox{\textwidth}{!}{
	\begin{tabular}{llll}
		\hline\noalign{\smallskip}
		Assumption & Convergence type & Convergence rate & References\\
		\noalign{\smallskip}\hline\noalign{\smallskip}
    	\multirow{4}{*}{convex $f$ and $g$}
			& ergodic & $\BigO{1/k}$ & \cite{boct2014convergence, Chambolle2011, Davis2014a, Davis2014, davis2016convergence,Goldstein2012, Monteiro2010a, Monteiro2011, Monteiro2010, Ouyang2014}, etc.\\
			& non-ergodic & $\BigO{1/k}$ & \cite{tran2017proximal, tran2018adaptive, TranDinh2015b, TranDinh2012a, valkonen2019inertial} and \textbf{this work}\\
			& non-ergodic & $\SmallOs{1/\sqrt{k}}$ & \cite{Davis2014a, Davis2014, davis2016convergence}\\
			& non-ergodic & $\min\big\{\BigO{1/k},$ & \textbf{this work}\\
			&& \hfill $\underline{o}\big(1/(k\sqrt{\log k})\big)\big\}$\\
		\noalign{\smallskip}\hline\noalign{\smallskip}
		\multirow{4}{*}{strongly convex $f$ or $g^*$}
			& ergodic & $\BigO{1/k^2}$ & \cite{Chambolle2011, Goldstein2012, Monteiro2010a, Monteiro2011, Ouyang2014}, etc.\\
			& non-ergodic & $\BigO{1/k^2}$ & \cite{tran2017proximal, tran2018adaptive, TranDinh2015b, TranDinh2012a, valkonen2019inertial} and \textbf{this work}\\
			& best-iterate & $\SmallO{1/k}$ & \cite{Davis2014, Davis2014b}\\
			& non-ergodic & $\min\big\{\BigO{1/k^2}, $ & \textbf{this work}\\
			&& \hfill $\underline{o}\big(1/(k^2\sqrt{\log{k}})\big)\big\}$\\
		\noalign{\smallskip}\hline
	\end{tabular}}
\vspace{-1ex}	
\end{table}

\beforepara
\paragraph{Paper organization}
The rest of this paper is organized as follows.
Section~\ref{sec:pre_results} reviews some preliminary tools used in the sequel.
Section~\ref{sec:new_pd_algs} develops a new algorithm for the general convex case, investigates its convergence rate guarantees, and applies it to solve constrained convex problems.
Section~\ref{sec:strong_convexity} studies the strongly convex case with a new algorithm and its convergence guarantees.
It also presents an application to constrained convex problems under semi-strongly convex assumption.
Section~\ref{sec:num_experiments} provides two illustrative numerical examples.
For the sake of presentation, all technical proofs of the results in the main text are deferred to the appendices.

\beforesec
\section{Basic Assumption and Optimality Conditions}\label{sec:pre_results}
\aftersec
\beforepara
\paragraph{Basic notation and concepts}
We work with Euclidean spaces $\R^p$ and $\R^n$ equipped with standard inner product $\iprods{\cdot,\cdot}$ and norm $\norm{\cdot}$.
For any nonempty, closed, and convex set $\Xc$ in $\R^p$, $\ri{\Xc}$ denotes the relative interior of $\Xc$ and $\delta_{\Xc}(\cdot)$ denotes the indicator of $\Xc$.
For any proper, closed, and convex function $f: \R^p\to\Rext$, $\dom{f} := \set{x \in\R^p \mid f(x) < +\infty}$ is its (effective) domain, $f^{\ast}(y) := \sup_{x}\{ \iprods{x, y} - f(x) \}$ denotes the Fenchel conjugate of $f$, $\partial{f}(x) := \{ w\in\R^p \mid f(y) - f(x) \geq \iprods{w, y-x},~\forall y\in\dom{f} \}$ stands for the subdifferential of $f$ at $x$, and $\nabla{f}$ is the gradient or subgradient of $f$. 

A function $f$ is called $M_f$-Lipschitz continuous on $\dom{f}$ with a Lipschitz constant $M_f \in [0, +\infty)$ if $\vert f(x) - f(y) \vert \leq M_f\norms{x - y}$ for all $x, y\in\dom{f}$.
If $f$ is differentiable on $\dom{f}$ and $\nabla{f}$ is Lipschitz continuous with a Lipschitz constant $L_f \in [0, +\infty)$, i.e., $\norms{\nabla{f}(x) - \nabla{f}(y)} \leq L_f\norms{x - y}$ for $x, y\in\dom{f}$, then we say that $f$ is $L_f$-smooth.
If $f(\cdot) - \frac{\mu_f}{2}\norms{\cdot}^2$ is still convex for some $\mu_f > 0$, then we say that $f$ is $\mu_f$-strongly convex with a strong convexity parameter $\mu_f$.
We also denote $\prox_{f}(x) := \mathrm{arg}\min_{y}\{ f(y) + \tfrac{1}{2}\norms{y-x}^2 \}$ the proximal operator of $f$.
For any $\rho > 0$, we have the following Moreau's identity \cite{Bauschke2011}:
\myeqf{eq:moreau_identity}{
	\prox_{f/\rho}(x) + \rho^{-1} \prox_{\rho f^{\ast}}(\rho x) = x.
}
We use $\BigO{\cdot}$, $\SmallO{\cdot}$, and $\LowerO{\cdot}$ to denote the order of complexity bounds as usual. 
With convergence rates, for two scalar sequences $\sets{u_k} \subseteq \R_{+}$ and $\sets{v_k} \subseteq \R_{++}$, we say that
\myeqfn{
	u_k = \BigO{v_k}~~\text{if}~~\limsup_{k \to \infty} \left(\frac{u_k}{v_k}\right) < +\infty\qquad \text{and}\qquad u_k = \SmallO{v_k}~~\text{if}~~\lim_{k \to \infty} \left(\frac{u_k}{v_k}\right) = 0.
}
In this paper, we further define a new $\underline{o}(\cdot)$ notation for convergence rates as follows:
\myeqf{eq:underline_o_def}{
	u_k = \underline{o}(v_k)~~\text{if}~~\liminf_{k \to \infty}\left(\frac{u_k}{v_k}\right) = 0.
}
That is, there is a subsequence $\sets{k_j} \subseteq \mathbb{N}$ such that $u_{k_j} = \SmallOs{v_{k_j}}$.

Our algorithms rely on the following assumption imposed on the problem \eqref{eq:com_cvx}:

\begin{assumption}\label{as:A0}
The functions $f$ and $g$ in \eqref{eq:com_cvx} are proper, closed, and convex. 
The solution set $\Xc^{\star}$ of \eqref{eq:com_cvx} is nonempty, and $0 \in \ri{\dom{g} - K\dom{f}}$.
\end{assumption}

Assumption~\ref{as:A0} is fundamental and required in any primal-dual method for theoretical convergence guarantees.
Since $\Xc^{\star}$ is nonempty, under Assumption~\ref{as:A0}, the strong duality holds, thus we have $F^{\star} + G^{\star} = 0$, and the solution set $\Yc^{\star}$ of the dual problem \eqref{eq:dual_prob} is also nonempty. 

\beforepara
\paragraph{Optimality conditions}
The optimality conditions of  \eqref{eq:com_cvx} and  \eqref{eq:dual_prob} are
\myeqn{
	\text{primal:}~~0 \in \partial{f}(x^{\star}) + K^{\top}\partial{g}(Kx^{\star}) ~~~~\text{or dual:}~~ 0 \in - K\partial{f^{\ast}}(-K^{\top}y^{\star}) + \partial{g^{\ast}}(y^{\star}).
}
These two conditions can be written into the following primal-dual optimality condition, which can also be derived from the min-max form \eqref{eq:min_max}:
\begin{equation*}
	\text{primal-dual:}~~0 \in K^{\top}y^{\star}  + \partial{f}(x^{\star})~~~\text{and}~~~0 \in -Kx^{\star} + \partial{g^{\ast}}(y^{\star}).
\end{equation*}

\beforepara
\paragraph{Gap function}
Let $\widetilde{\Lc}(x,y) := f(x) + \iprods{Kx, y} - g^{\ast}(y)$ be defined by \eqref{eq:min_max} and $\Xc$ and $\Yc$ be given nonempty, closed, and convex sets such that $\Xc^{\star}\cap\Xc\neq\emptyset$ and $\Yc^{\star}\cap\Yc\neq\emptyset$.
We define a gap function $\Gc_{\Xc\times\Yc}(\cdot)$ as follows:
\myeq{eq:gap_func}{
	\Gc_{\Xc\times\Yc}(x, y) := \displaystyle\sup_{(\hat{x}, \hat{y}) \in \Xc \times \Yc} \Big\{\widetilde{\Lc}(x, \hat{y}) - \widetilde{\Lc}(\hat{x}, y) \Big\} 
	= \displaystyle\sup_{\hat{y} \in \Yc} \widetilde{\Lc} (x, \hat{y}) - \displaystyle\inf_{\hat{x} \in \Xc} \widetilde{\Lc} (\hat{x}, y).
}
Then, we immediately have
\myeqn{
	\Gc_{\Xc\times\Yc} (x, y) \geq \widetilde{\Lc} (x, y^\star) - \widetilde{\Lc} (x^\star, y) \geq \widetilde{\Lc}(x^\star, y^\star) - \widetilde{\Lc}(x^\star, y^\star) = 0,
}
where $(x^{\star}, y^{\star})\in\Xc^{\star}\times\Yc^{\star}$ is a primal-dual solution of \eqref{eq:com_cvx} and \eqref{eq:dual_prob}, i.e., a saddle-point of $\widetilde{\Lc}$. 
Moreover, $\Gc_{\Xc\times\Yc}$ vanishes at any saddle point $(x^\star, y^\star)$.
Thus this gap function can be considered as a measure of optimality for both \eqref{eq:com_cvx} and \eqref{eq:dual_prob}.

\beforepara
\paragraph{Constrained reformulation and merit function}
The primal problem \eqref{eq:com_cvx} can be reformulated into the following equivalent constrained setting:
\myeq{eq:constr_reform}{
	F^{\star} := \min_{x \in \R^p,\ r \in \R^n} \Big\{ F(x, r) := f(x) + g(r) ~~\textrm{s.t.}~~ Kx - r = 0 \Big\}.
}
Let $\Lc(x, r, y) :=  f(x) +  g(r) + \iprods{Kx - r, y}$ be the Lagrange function associated with \eqref{eq:constr_reform}, where $y \in\R^n$ is the corresponding Lagrange multiplier, and $\widetilde{\Lc}(x, y) := f(x) + \iprods{Kx, y} - g^{\ast}(y)$ be defined by \eqref{eq:min_max}.
Since $g^{\ast}(y) := \sup_{r\in\R^n}\set{ \iprods{y, r} - g(r)}$, we can show that, for any $r \in \R^n$, one has
\myeq{eq:lag_func12}{
\widetilde{\Lc}(x, y) \leq f(x) +  g(r) + \iprods{Kx - r, y} = \Lc(x, r, y).
}
Moreover, $\widetilde{\Lc}(x, y) =  \Lc(x, r, y)$ if and only if $y\in\partial{g}(r)$ or equivalently $r\in\partial{g^{\ast}}(y)$.

Together with $\Lc$, we define an augmented Lagrangian $\Lc_{\rho}$ as
\myeq{eq:aug_Lag}{
\Lc_{\rho}(x,r, y) := \Lc(x, r, y) + \frac{\rho}{2}\norms{Kx - r}^2 = f(x) + g(r) + \phi_{\rho}(x, r, y),
}
where $\phi_{\rho}(x, r, y) := \iprods{Kx - r, y} + \frac{\rho}{2}\norms{Kx - r}^2$ and $\rho > 0$ is a penalty parameter.
Note that the term $\frac{\rho}{2}\norms{Kx - r}^2$ can be viewed as a smoothed approximation of $\delta_{\set{(x,r) \mid Kx - r = 0}}(x,r)$, the indicator of $\set{(x,r) \mid Kx - r = 0}$.
The function $\Lc_{\rho}$ will serve as a \mytbi{merit function} to develop our algorithms in the sequel.

\beforesec
\section{A New Primal-Dual Algorithm for General Convex Case}\label{sec:new_pd_algs}
\aftersec
In this section, we develop a novel primal-dual algorithm to solve \eqref{eq:com_cvx} and its dual form \eqref{eq:dual_prob} with fast convergence rate guarantees, where $f$ and $g$ are both \mytbi{merely convex}.

\beforesubsec
\subsection{\bf \bf Algorithm derivation and one-iteration analysis}
\aftersubsec
Our main idea is to combine four  techniques in one: alternating minimization, linearization, acceleration, and homotopy. 
While each individual technique is classical, their entire combination is \mytbi{new}.
At the iteration $k \geq 0$, given $x^k$, $\tilde{x}^k$, $r^k$, and  $\tilde{y}^k$, we update
\myeq{eq:pd_scheme1}{
\left\{\begin{array}{lcl}
	\hat{x}^k & := & (1-\tau_k)x^k + \tau_k\tilde{x}^k,\vspace{1ex}\\
	r^{k+1} & := & \prox_{g/\rho_k}\big( \tilde{y}^k/\rho_k + K\hat{x}^k\big),\vspace{1ex}\\
	x^{k+1} & := & \prox_{\beta_kf}\big(\hat{x}^k -  \beta_k\nabla_x\phi_{\rho_k}(\hat{x}^k, r^{k+1}, \tilde{y}^k)\big),\vspace{1ex}\\
	\tilde{x}^{k+1} & := & \tilde{x}^k + \frac{1}{\tau_k}(x^{k+1} - \hat{x}^k),\vspace{1ex}\\
	\tilde{y}^{k+1} & := & \tilde{y}^k + \eta_k \left[Kx^{k+1} - r^{k+1} - (1-\tau_k)(Kx^k - r^k)\right].
\end{array}\right.
}
We now explain each step of the scheme \eqref{eq:pd_scheme1} as follows:
\begin{itemize}
\item Line 2 and line 3 of \eqref{eq:pd_scheme1} alternatively minimize the merit function $\Lc_{\rho}$ w.r.t. $r$ and $x$ to obtain $r^{k+1}$ and $x^{k+1}$, respectively. 
However, since the subproblem in $x^{k + 1}$ is difficult to solve, we linearize the coupling term $\frac{\rho}{2}\norms{Kx - r}^2$ as 
\myeqn{
\begin{array}{ll}
	\frac{\rho_k}2 \norms{Kx - r^{k + 1}}^2 & \approx  \frac{\rho_k}2 \norms{K\hat{x}^k - r^{k + 1}}^2\vspace{1ex}\\
	& + {~} \frac{\rho_k}2 \iprods{\nabla_x \norms{K\hat{x}^k - r^{k + 1}}^2, x - \hat{x}^k} + \frac1{2\beta_k} \norms{x - \hat{x}^k}^2,
\end{array}
}
so that we can simply use the proximal operator of $f$ as in line 3.
\item Line 1 and line 4  update $\hat{x}^k$ and $\tilde{x}^{k + 1}$, respectively to accelerate the primal iterates using Nesterov's acceleration strategy \cite{Nesterov1983}.
\item Line 5 updates the dual variable $\tilde{y}^{k + 1}$ as in augmented Lagrangian methods.
\end{itemize}
All the parameters $\tau_k\in (0, 1]$, $\rho_k > 0$, $\beta_k > 0$, and $\eta_k > 0$ will be updated in a homotopy fashion.
We will explicitly provide update rules for these parameters in Algorithm~\ref{alg:A1} based on our convergence analysis.

The following lemma provides a key estimate on the difference $\Lc_{\rho_{k - 1}} (x^k, r^k, y) - \Lc(x, r, \bar{y}^k)$ to prove Theorems~\ref{th:convergence_guarantee1} and \ref{th:faster_convergence_rate}. 
The proof is deferred to Appendix \ref{apdx:le:key_bound1}.

\begin{lemma}\label{le:key_bound1}
Let $(x^k, \hat{x}^k, \tilde{x}^k, r^k, \tilde{y}^k)$ be  updated by \eqref{eq:pd_scheme1} with $\rho_k > \eta_k$ and $\bar{y}^{k + 1} := (1 - \tau_k)\bar{y}^k + \tau_k \left[\tilde{y}^k + \rho_k (K\hat{x}^k - r^{k + 1})\right]$.
Then, for any $(x, r, y)\in\R^p\times \R^n\times\R^n$, the following inequality holds:
\myeq{eq:key_bound1}{
\hspace{-4ex}
\arraycolsep=0.3em
\begin{array}{ll}
	&\left[\Lc_{\rho_k} (x^{k + 1}, r^{k + 1}, y) - \Lc(x, r, \bar{y}^{k + 1})\right] \leq (1 - \tau_k)\left[ \Lc_{\rho_{k - 1}} (x^k, r^k, y) - \Lc(x, r, \bar{y}^k)\right] \vspace{1ex}\\
	& \quad\quad\quad + {~} \frac{\tau_k^2}{2\beta_k} \left[ \norms{\tilde{x}^k - x}^2 - \norms{\tilde{x}^{k + 1} - x}^2\right] + \frac1{2\eta_k} \left[ \norms{\tilde{y}^k - y}^2 - \norms{\tilde{y}^{k + 1} - y}^2\right] \vspace{1ex}\\
	& \quad\quad\quad - {~} \frac12 \left(\frac1{\beta_k} - \frac{\rho_k^2 \norms{K}^2}{\rho_k - \eta_k}\right)\norms{x^{k + 1} - \hat{x}^k}^2 - \frac{(1 - \tau_k)}{2}\left[\rho_{k - 1} - (1 - \tau_k)\rho_k\right] \norms{Kx^k - r^k}^2.
\end{array}
\hspace{-3ex}
}
\end{lemma}

\subsection{The complete algorithm}\label{subsec:alg1}
To transform our scheme \eqref{eq:pd_scheme1} into a primal-dual format, we first eliminate $r^k$ and $r^{k+1}$. 
By Moreau's identity \eqref{eq:moreau_identity}, we have
\myeq{eq:rk_step}{
	r^{k+1} = \tfrac{1}{\rho_k}\big( \tilde{y}^k + \rho_kK\hat{x}^k - y^{k+1}\big), \quad \text{where}\quad y^{k+1} := \prox_{\rho_kg^{\ast}}\big(\tilde{y}^k + \rho_kK\hat{x}^k\big).
}
Now, from the definition of $\phi_\rho$ in \eqref{eq:aug_Lag}, we can write
\myeqn{
	\nabla_x\phi_{\rho_k}(\hat{x}^k, r^{k+1}, \tilde{y}^k) = K^{\top}\big(\tilde{y}^k + \rho_kK\hat{x}^k - \rho_kr^{k+1}\big) = K^{\top}y^{k+1}.
}
Substituting this expression into line 3 of \eqref{eq:pd_scheme1}, we can eliminate $r^{k+1}$. 
Next, we combine line 1 and line 4 of \eqref{eq:pd_scheme1} to obtain $\hat{x}^{k+1} = x^{k+1} + \frac{\tau_{k+1}(1-\tau_k)}{\tau_k}(x^{k+1} - x^k)$.
Finally,  substituting $r^k$ using \eqref{eq:rk_step} into line 5 of \eqref{eq:pd_scheme1}, we can express $\tilde{y}^{k + 1}$ as
\myeq{eq:update_yhat}{
\arraycolsep=0.27em
\begin{array}{lcl}
	\tilde{y}^{k+1} &= & \tilde{y}^k + \eta_kK\big(x^{k+1} - \hat{x}^k - (1-\tau_k)(x^k - \hat{x}^{k-1})\big) \vspace{1ex}\\
	& & - {~} \frac{\eta_k}{\rho_k}(\tilde{y}^k - y^{k+1}) + \frac{\eta_k(1-\tau_k)}{\rho_{k-1}}(\tilde{y}^{k-1} - y^k).
\end{array}
}
In addition to \eqref{eq:pd_scheme1}, we also update $y$ using the following weighted averaging step:
\myeq{eq:dual_avarging}{
	\bar{y}^{k+1} := (1-\tau_k)\bar{y}^k + \tau_ky^{k+1},
}
where $y^{k + 1}$ is defined in \eqref{eq:rk_step}. This is consistent with the condition in Lemma \ref{le:key_bound1}. 

For the parameters, as guided by Lemma~\ref{le:key_bound1}, we propose the following update:
\myeq{eq:update_rule1}{
	\tau_k := \frac{c}{k+c}, \qquad \rho_k := \frac{\rho_0}{\tau_k}, \qquad \beta_k := \frac{\gamma}{\norms{K}^2\rho_k} \quad \text{and} \quad \eta_k := (1-\gamma)\rho_k,
}
where $c \geq 1$, $\gamma \in (0, 1)$, and $\rho_0 > 0$ are given. 

In summary, we describe the complete primal-dual algorithm as in Algorithm \ref{alg:A1}.

\begin{algorithm}[htbp!]\caption{(New Primal-Dual Algorithm for \eqref{eq:com_cvx} and \eqref{eq:dual_prob}: General Convex Case)}\label{alg:A1}
\normalsize
\begin{algorithmic}[1]
   \State{\bfseries Initialization:} Choose $x^0 \in\R^p$, $y^0\in\R^n$, $\rho_0 > 0$, $c \geq 1$, and $\gamma \in (0, 1)$.
   \vspace{0.5ex}
   \State Set $\tau_0 := 1$, ~$x^{-1} := \hat{x}^0 := x^0$, and $\tilde{y}^{-1} := \tilde{y}^0 := \bar{y}^0 := y^0$.
   \vspace{0.5ex}
   \State{\bfseries For $k := 0, 1, \cdots, k_{\max}$ do}
   \vspace{0.5ex}   
   \State\hspace{3ex}\label{step:i1} 
    Update $\rho_k := \frac{\rho_0}{\tau_k}$, ~$\beta_k := \frac{\gamma}{\norms{K}^2\rho_k}$, ~$\eta_k := (1-\gamma)\rho_k$, and $\tau_{k+1} := \frac{c}{k + c + 1}$.
    \vspace{0.5ex}   
	\State\hspace{3ex}\label{step:i2} Update the primal-dual step:
	\vspace{-1ex}
	\myeqn{
		\arraycolsep=0.27em
		\left\{\begin{array}{lcl}
			y^{k+1} & := & \prox_{\rho_kg^{\ast}}\big(\tilde{y}^k + \rho_kK\hat{x}^k\big),\vspace{1ex}\\
			x^{k+1} & := & \prox_{\beta_kf}\big(\hat{x}^k - \beta_kK^{\top}y^{k+1}\big),\vspace{1ex}\\
			\hat{x}^{k+1} &:= & x^{k+1} + \frac{\tau_{k+1}(1-\tau_k)}{\tau_k} (x^{k+1} - x^k).
		\end{array}\right.
		\vspace{-1ex}	
	}
    \State\hspace{3ex}\label{step:i3} Update the intermediate dual step:
    \vspace{-1ex}
	\myeqn{
		\arraycolsep=0.27em
		\begin{array}{lcl}
			\tilde{y}^{k+1} & := & \tilde{y}^k + \eta_k K\big[x^{k+1} - \hat{x}^k - (1-\tau_k)(x^k - \hat{x}^{k-1})\big]\vspace{1ex}\\
			& & + {~} (1 - \gamma)\left[y^{k + 1} - \tilde{y}^k - \frac{\tau_{k - 1} (1 - \tau_k)}{\tau_k} (y^k - \tilde{y}^{k - 1})\right].
		\end{array}
		\vspace{-1ex}	
	}
	\State\hspace{3ex}\label{step:i4} Update the dual averaging step:	 $\bar{y}^{k+1} := (1-\tau_k)\bar{y}^k + \tau_ky^{k+1}$.
    \vspace{0.5ex}   
   \State{\bfseries EndFor}
\end{algorithmic}
\end{algorithm}

\noindent Let us highlight the following features of Algorithm~\ref{alg:A1}.
\begin{itemize}
\item Algorithm~\ref{alg:A1} updates its parameters at Step~\ref{step:i1} dynamically. 
The update of $\tau_k$ is often seen in Nesterov's accelerated-based schemes.
The penalty parameter $\rho_k$ is not fixed, but is updated in a homotopy fashion and also different from the dual step-size $\eta_k$.
The dual update at Step~\ref{step:i3} is completely new and depends on three consecutive iterations.
All these properties are fundamentally different from existing primal-dual and augmented Lagrangian-based methods. 

\item We use two parameters $\gamma \in (0, 1)$ and $\rho_0 > 0$ to balance the primal term $\norms{x^0 - x^{\star}}^2$ and dual term $\norms{y^0 - y^{\star}}^2$ in the bound \eqref{eq:convergence_bound1} of Theorem~\ref{th:convergence_guarantee1} below.
Note that our update leads to $\rho_k\beta_k\norms{K}^2 = \gamma < 1$, which is the same as the parameter condition in the Chambolle-Pock primal-dual method \cite{Chambolle2011}.

\item The per-iteration complexity of Algorithm~\ref{alg:A1} is essentially the same as in existing primal-dual methods. 
It requires one $\prox_{\rho_kg^{\ast}}$, one $\prox_{\beta_kf}$, one $Kx$, and one $K^{\top}y$.
The matrix-vector multiplication at Step~\ref{step:i3} can be eliminated if we store $Kx^k$ and $Kx^{k + 1}$, and use the last line of Step~\ref{step:i2} to compute $K\hat{x}^k$.
\end{itemize}

\beforesubsec
\subsection{\bf Convergence analysis}\label{subsec:convergence_thm_general}
\aftersubsec
The following theorem states convergence guarantees of Algorithm~\ref{alg:A1} under Assumption~\ref{as:A0} with $c = 1$ without any smoothness or strong convexity assumption. Its proof is given in  Appendix~\ref{apdx:th:convergence_guarantee1}.

\begin{theorem}[$\BigO{1/k}$ convergence rates when $c = 1$]\label{th:convergence_guarantee1}
Let $\sets{(x^k, \bar{y}^k)}$ be generated by Algorithm~\ref{alg:A1} with $c := 1$, and $\widetilde{\Lc}$ be defined by \eqref{eq:min_max}. Then, under Assumption~\ref{as:A0}, the following bound is valid for any given $x, x^0\in\R^p$ and $y, y^0\in\R^n$:
\myeq{eq:convergence_bound1}{
	\widetilde{\Lc}(x^k, y) - \widetilde{\Lc}(x, \bar{y}^k) \leq \frac{1}{2k}\left[\frac{\rho_0\norms{K}^2\norms{x^0 - x}^2}{\gamma} + \frac{\norms{y^0 - y}^2}{(1-\gamma)\rho_0} \right].
}
Furthermore, the following statements hold:
\begin{itemize}
\item[\textnormal{(a)}] \textnormal{(Semi-ergodic convergence on the primal-dual gap).} The gap function $\Gc_{\Xc\times\Yc}$ defined by \eqref{eq:gap_func} satisfies the following bound for all $k \geq 1$:
\myeq{eq:primal_dual_bound1}{
	0 \leq \Gc_{\Xc\times\Yc} (x^k, \bar{y}^k) \leq \frac1{2k}  \sup_{(x, y) \in \Xc \times \Yc}  \left[\frac{\rho_0\norms{K}^2\norms{x^0 - x}^2}{\gamma} + \frac{\norms{y^0 - y}^2}{(1-\gamma)\rho_0} \right].{\!\!\!\!\!\!\!}
}
Hence, the primal-dual gap sequence $\set{\Gc_{\Xc\times\Yc}(x^k, \bar{y}^k)}$ converges to zero at $\BigO{1/k}$ rate in semi-ergodic sense, i.e., non-ergodic in $x^k$ and ergodic in $\bar{y}^k$. 

\item[\textnormal{(b)}] \textnormal{(Non-ergodic convergence on the primal objective).} 
If $g$ is $M_g$-Lipschitz continuous on $\dom{g}$ and $x^\star$ is an optimal solution of \eqref{eq:com_cvx}, then,  for $k \geq 1$, the primal objective residual based on the last-iterate sequence $\set{x^k}$ satisfies:
\myeq{eq:primal_bound1}{
	0 \leq F(x^k) - F^{\star} \leq \frac{1}{2k}\left[\frac{\rho_0\norms{K}^2\norms{x^0 - x^{\star}}^2}{\gamma}  + \frac{D^2_g}{(1-\gamma)\rho_0} \right],
}
where $D_g := \sup\set{ \norms{y^0 - y} \mid \norms{y} \leq M_g}$. 
Hence, $\sets{F(x^k)}$ converges to the primal optimal value $F^{\star}$ of \eqref{eq:com_cvx} at $\BigO{1/k}$ rate in non-ergodic sense.

\item[\textnormal{(c)}] \textnormal{(Ergodic convergence on the dual objective).} 
If $f^{\ast}$ is $M_{f^{\ast}}$-Lipschitz continuous on $\dom{f^{\ast}}$ and $y^\star$ is an optimal solution of \eqref{eq:dual_prob}, then, for all $k \geq 1$, the dual objective residual based on the averaging sequence $\sets{\bar{y}^k}$ satisfies:
\myeq{eq:dual_bound1}{
	0 \leq G(\bar{y}^k) - G^{\star} \leq \frac{1}{2k}\left[\frac{\rho_0\norms{K}^2 D_{f^{\ast}}^2 }{\gamma}  + \frac{ \norms{y^0 - y^{\star}}^2}{(1-\gamma)\rho_0} \right],
}
where $D_{f^{\ast}} := \sup\set{\norms{x^0 - x} \mid \norm{x} \leq M_{f^{\ast}}}$.
Hence, $\sets{G(\bar{y}^k)}$ converges to the dual optimal value $G^{\star}$ of \eqref{eq:dual_prob} at $\BigO{1/k}$ rate in ergodic sense. 
\end{itemize}
\end{theorem}

\begin{remark}[\textit{Optimal rate}]\label{rmk:optimal_rate}
It was shown in \cite{li2016accelerated,woodworth2016tight} that, under Assumption~\ref{as:A0}, the rate $\BigO{1/k}$ is \textbf{optimal}, in the sense that for any algorithm $\mathcal{A}$ for solving \eqref{eq:com_cvx}, in order to achieve the bound $F(x^k) - F^\star \leq \varepsilon$, there exists an instance of $f$ and $g$ with their arguments' dimensions $p$ and $n$ dependent on $\varepsilon$, such that $\mathcal{A}$ makes $\LowerO{1/\varepsilon}$ queries to the first-order oracle of $f$ and $g$ $($e.g., $f(x)$, $\nabla f(x)$, or $\prox_{\rho f}(x)$$)$. 
In other words, the convergence rate of $\mathcal{A}$ cannot exceed $\BigO{1/k}$ rate under Assumption~\ref{as:A0} when the problem dimension $p$ is much larger  than the number of iterations $k$, i.e., $k\leq \BigO{p}$.
Consequently, Algorithm~\ref{alg:A1} indeed achieves \textbf{optimal} convergence rate.
\end{remark}

\begin{remark}[\textit{Symmetry}]
Since the primal-dual problems \eqref{eq:com_cvx} and \eqref{eq:dual_prob} are symmetric, to obtain a non-ergodic convergence rate on the dual problem \eqref{eq:dual_prob}, we could apply Algorithm~\ref{alg:A1} to the dual-primal pair instead of the primal-dual pair.
\end{remark}

If we choose $c > 1$, then Algorithm \ref{alg:A1} still converges. 
In fact, it  achieves the same $\BigO{1/k}$ and a potentially faster\footnote{In fact, our numerical experiments in Section \ref{sec:num_experiments} show  significantly faster empirical convergence rates of Algorithm~\ref{alg:A1} when we use the parameter update rules \eqref{eq:update_rule1} with $c > 1$.} $\underline{o}\big(1/(k\sqrt{\log{k}})\big)$ convergence rate on the primal objective residual, as shown in Theorem~\ref{th:faster_convergence_rate}, whose proof is given in Appendix~\ref{apdx:th:faster_convergence_rate}.

\begin{theorem}[$\BigO{1/k}$ and  $\underline{o}\big(1/(k\sqrt{\log{k}})\big)$ convergence rates when $c > 1$]\label{th:faster_convergence_rate}
Let $\sets{x^k}$ be  generated by Algorithm~\ref{alg:A1} with $c > 1$. 
Let $\widetilde{\Lc}$ be defined by \eqref{eq:min_max} and $y^\star$ be an optimal solution of \eqref{eq:dual_prob}. 
Then, under Assumption~\ref{as:A0}, for any $k\geq 0$, we have
\myeq{eq:primal_dual_bound1_new}{
	0 \leq \widetilde{\Lc}(x^k, y^\star) - F^\star \leq \frac{R_0^2}{k + c - 1} ~~~\text{and}~~~ \liminf_{k \to \infty}\, k\log{(k)}\big[\widetilde{\Lc}(x^k, y^{\star}) - F^{\star}\big] = 0,
}
where $R_0^2 := (c - 1)\left[F(x^0) - F^\star\right] + \frac{c}2 \left[\frac{\rho_0 \norms{K}^2}\gamma \norms{x^0 - x^\star}^2 + \frac1{(1 - \gamma)\rho_0} \norms{y^0 - y^\star}^2\right]$. 

Moreover, if $g$ is $M_g$-Lipschitz continuous on $\dom{g}$, then the primal last-iterate sequence $\sets{x^k}$ satisfies the following statements for all $k \geq 0$:
\myeq{eq:primal_bound1_new2}{
	0 \leq F(x^k) - F^{\star} \leq \frac{R_1^2}{k + c - 1}\quad \text{and}\quad  \displaystyle\liminf_{k \to \infty} \, k\sqrt{\log{k}}\big[F(x^k) - F^\star \big] = 0,
}
where $R_1^2 := R_0^2 + \sqrt{2c/\rho_0}(\norms{y^\star} + M_g)R_0$. 
Hence, $\sets{F(x^k)}$ converges to the primal optimal value $F^{\star}$ of \eqref{eq:com_cvx} at both $\BigO{1/k}$ and $\underline{o}\big(1/(k\sqrt{\log k})\big)$ convergence rates in non-ergodic sense, where $\underline{o}(\cdot)$ is defined in \eqref{eq:underline_o_def}.
\end{theorem}

\begin{remark}\label{re:comparison1}
The $\underline{o}\big(1/(k\sqrt{\log{k}})\big)$ rate does not contradict our discussion in Remark \ref{rmk:optimal_rate}, since our problem dimensions $p$ and $n$ are fixed, while $k$ can be sufficiently large. 
\end{remark}

\subsection{Application to constrained problems}\label{subsec:constr_prob}
Let us apply Algorithm~\ref{alg:A1} to solve the following nonsmooth constrained convex optimization problem:
\myeq{eq:constr_cvx}{
	F^{\star} := \min_{x\in\R^p}\Big\{ F(x) := f(x) + \psi(x) \quad \text{s.t.} \quad Kx = b \Big\},
}
where $f$ and $K$ are defined as in \eqref{eq:com_cvx}, $\psi$ is proper, closed, and convex, and $b\in\R^n$. 
This problem is a special case of \eqref{eq:com_cvx} where $f$ is replaced by $f + \psi$, and $g(u) := \delta_{\set{b}}(u)$, the indicator of $\set{b}$. 
In this case, $g^{\ast} (y) = \iprods{b, y}$, and the last condition of Assumption~\ref{as:A0} reduces to the Slater condition: $\ri{\dom{f}\cap\dom{\psi}}\cap\set{ x \mid Kx = b}\neq\emptyset$. 
In addition, we require the following assumption on the new objective term $\psi$:

\begin{assumption}\label{as:A1psi}
	The function $\psi$ in \eqref{eq:constr_cvx} is convex and $L_\psi$-smooth.
\end{assumption}

\noindent We specify Algorithm~\ref{alg:A1} to solve \eqref{eq:constr_cvx} and its dual problem as follows:
\myeq{eq:constr_alg1}{
\arraycolsep=0.3em
\left\{\begin{array}{lcl}
		y^{k + 1}	& :=	& \tilde{y}^k + \rho_k (K\hat{x}^k - b),\vspace{1ex}\\
		x^{k + 1} 	& :=	& \prox_{\beta_kf}\left(\hat{x}^k - \beta_k \left[K^{\top}y^{k + 1} + \nabla{\psi}(\hat{x}^k)\right]\right),\vspace{1ex}\\
		\hat{x}^{k + 1}	& :=	& x^{k + 1} + \frac{\tau_{k + 1} (1 - \tau_k)}{\tau_k} (x^{k + 1} - x^k),\vspace{1ex}\\
		\tilde{y}^{k + 1}	& :=	& \tilde{y}^k + \eta_k \left[K\left(x^{k + 1} - (1 - \tau_k)x^k\right) - \tau_k b\right],\vspace{1ex}\\
		\bar{y}^{k + 1}	& :=	& (1 - \tau_k)\bar{y}^k + \tau_k y^{k + 1},
	\end{array}\right.
}
where all the parameters are updated as in Algorithm~\ref{alg:A1} with a small modification $\beta_k := \gamma/(\norms{K}^2 \rho_k + \gamma L_\psi)$. 
Its convergence guarantee is summarized in the following corollary, whose proof is given in Appendix \ref{apdx:th:constr_alg_convergence}.

\begin{corollary}\label{th:constr_alg_convergence}
Let $\set{(x^k, \bar{y}^k)}$ be generated by scheme \eqref{eq:constr_alg1} to solve \eqref{eq:constr_cvx} and its dual problem under Assumptions \ref{as:A0} and \ref{as:A1psi}. 
Let $(x^\star, y^\star)$ be a pair of primal-dual optimal solution of \eqref{eq:constr_cvx}.
If we choose $c := 1$, then,  for all $k\geq 1$, we have  the following primal convergence rate guarantee:
\myeq{eq:constr_alg_convergence1}{
	\vert F(x^k) - F^{\star} \vert \leq \dfrac{R_0^2}{2k}\quad \text{and}\quad \norms{Kx^k - b} \leq \dfrac{R_0^2}{2k},
}
where $R_0^2 := \frac{\rho_0 \norms{K}^2 + \gamma L_\psi}{\gamma}\norms{x^0 - x^\star}^2 + \frac{1}{(1 - \gamma)\rho_0}(2\norms{y^\star} + \norms{y^0} + 1)^2$.
Hence, the objective residual and the feasibility violation both converge to zero at $\BigO{1/k}$ non-ergodic rate.

If, in addition, $\dom{F}$ is bounded, then we have the dual convergence guarantee:
\myeq{eq:constr_alg_convergence1_dual}{
	G(\bar{y}^k) - G^{\star} \leq \dfrac{1}{2k}\left[\frac{(\rho_0\norms{K}^2 + \gamma L_{\psi})\Dc_{F}^2}{\gamma} + \frac{\norms{y^0 - y^{\star}}^2}{(1-\gamma)\rho_0}\right],
}
where $\Dc_F := \sup\set{ \norms{x - x^0} \mid x \in\dom{F}} < +\infty$. 

If we choose $c > 1$ in the variant of Algorithm~\ref{alg:A1} for solving \eqref{eq:constr_cvx}, then the $\BigO{1/k}$ non-ergodic rate bounds on $\vert F(x^k) - F^{\star} \vert$ and $\norms{Kx^k - b}$ still hold, and
\myeq{eq:constr_alg_convergence_faster1}{
	\liminf_{k\to\infty} \, k\sqrt{\log{k}}\left[ \vert F(x^k) - F^{\star} \vert   + \norms{Kx^k - b} \right] = 0.
}
Hence, the objective residual and  feasibility violation sequences both converge to zero at  both $\BigO{1/k}$ and $\underline{o}\big(1/(k\sqrt{\log k})\big)$ non-ergodic convergence rates.
\end{corollary}

\begin{remark}\label{re:comparison3}
The $\BigO{1/k}$ rate results of Corollary \ref{th:constr_alg_convergence} are similar to  \cite{tran2017proximal,TranDinh2012a}. 
However, \cite{tran2017proximal} studied only primal methods for constrained convex problems using quadratic penalty framework and alternating minimization techniques without updating dual variables, and thus does not have convergence guarantee on the dual problem. 
The other work \cite{TranDinh2012a} relies on a different approach called smoothing techniques and excessive gap framework introduced in \cite{Nesterov2005d}. 
\end{remark}

\begin{remark}\label{re:general_constraints}
	We can extend the scheme \eqref{eq:constr_alg1} to solve  \eqref{eq:constr_cvx} with  general linear constraint $Kx - b \in \Sc$ instead of $Kx - b = 0$, where  $\Sc$ is  a nonempty, closed, and convex set in $\R^n$.
This constraint covers linear inequality constraints as special cases.
One simple trick is to introduce a slack variable $s$ and reformulating this constraint into $Kx - s = b$ and $s\in\Sc$.
Next, we replace the objective function $F(x) = f(x) + \psi(x)$ in \eqref{eq:constr_cvx} by $F(x,s) := f(x) + \psi(x) + \delta_{\Sc}(s)$, where $\delta_{\Sc}(s)$ is the indicator of $\Sc$.
Then, we can apply \eqref{eq:constr_alg1} to solve the resulting problem in $x$ and $s$.
In this case, our new scheme requires projection onto $\Sc$.
As another option, we can adopt the approach in \cite{tran2017proximal} to handle $Kx - b \in \Sc$ directly without reformulation.
We omit this extension to avoid overloading the paper.
\end{remark}

\beforesec
\section{A New Primal-Dual Method for Strongly Convex Case}\label{sec:strong_convexity}
\aftersec
In this section, we consider a special case of problem \eqref{eq:com_cvx}, where $f$ is strongly convex. 
More precisely, we impose the following assumption. 

\begin{assumption}\label{as:A2}
The function $f$ in \eqref{eq:com_cvx} is strongly convex with a strong convexity parameter $\mu_f > 0$, but not necessarily smooth.
\end{assumption}

\beforesubsec
\subsection{\bf Algorithm derivation and one-iteration analysis}
\aftersubsec
We follow the same diagram as in Section~\ref{sec:new_pd_algs}, but replacing Nesterov's accelerated step \cite{Nesterov1983} by Tseng's scheme \cite{tseng2008accelerated}, which allows us  to achieve a $\BigO{1/k^2}$  non-ergodic convergence rate. 
With this modification, we now describe our primal-dual scheme for solving \eqref{eq:com_cvx}-\eqref{eq:dual_prob}:
\myeq{eq:pd_scheme3}{
\arraycolsep=0.27em
\left\{\begin{array}{lcl}
	\hat{x}^k & := & (1 - \tau_k)x^k + \tau_k \tilde{x}^k,\vspace{1ex}\\
	r^{k + 1} & := & \prox_{g/\rho_k} (\tilde{y}^k/\rho_k + K\hat{x}^k),\vspace{1.25ex}\\		
	\tilde{x}^{k + 1} & := & \prox_{(\beta_k/\tau_k)f} \big(\tilde{x}^k - \frac{\beta_k}{\tau_k} \nabla_x \phi_{\rho_k} (\hat{x}^k, r^{k + 1}, \tilde{y}^k)\big),\vspace{1.25ex}\\
	x^{k + 1} & := & \prox_{f/(\rho_k \norms{K}^2)} \big(\hat{x}^k - \tfrac1{\rho_k\norms{K}^2} \nabla_x \phi_{\rho_k} (\hat{x}^k, r^{k + 1}, \tilde{y}^k)\big),\vspace{1.25ex}\\
	\tilde{y}^{k + 1} & := & \tilde{y}^k + \eta_k \big[Kx^{k + 1} - r^{k + 1} - (1 - \tau_k)(Kx^k - r^k)\big].
\end{array}\right.
}
The parameters $\tau_k$,\ $\rho_k$,\ $\beta_k$, and $\eta_k$ will be specified later based on our analysis.

We first analyze one iteration of the primal-dual scheme \eqref{eq:pd_scheme3} in the following lemma to obtain a recursive estimate. Its proof can be found in Appendix \ref{apdx:le:key_est_scvx2}.

\begin{lemma}\label{le:key_est_scvx2}
Let $(x^k, \hat{x}^k, \tilde{x}^k, r^k, \tilde{y}^k)$ be generated by \eqref{eq:pd_scheme3}, and $\bar{y}^{k + 1} := (1 - \tau_k)\bar{y}^k + \tau_k \left[\tilde{y}^k + \rho_k (K\hat{x}^k - r^{k + 1})\right]$. 
Assume that  $\rho_k > \eta_k$ and $\rho_k\beta_k\norms{K}^2 < 1$.
Then, for any $(x, r, y) \in \R^p \times \R^n \times \R^n$, it holds that
\myeq{eq:key_est_scvx2}{
\hspace{-2ex}
\arraycolsep=0.2em
\begin{array}{rl}
	\big[\Lc_{\rho_k} {\!\!\!}& (x^{k + 1}, r^{k + 1}, y) - \Lc(x, r, \bar{y}^{k + 1}) \big] \leq (1 - \tau_k)\big[\Lc_{\rho_{k - 1}}(x^k, r^k, y) - \Lc(x, r, \bar{y}^k)\big]\vspace{1.2ex}\\
	& + {~} \frac{\tau_k^2}{2\beta_k} \norms{\tilde{x}^k - x}^2 - \frac{\tau_k (\tau_k + \beta_k \mu_f)}{2\beta_k} \norms{\tilde{x}^{k + 1} - x}^2 
	+ \frac1{2\eta_k} \left[ \norms{\tilde{y}^k - y}^2 - \norms{\tilde{y}^{k + 1} - y}^2\right]\vspace{1.2ex}\\
	& - {~} \frac{\rho_k}2 \Big(1 - \rho_k \beta_k \norms{K}^2 - \frac{\eta_k}{\rho_k - \eta_k}\Big)\norms{K(x^{k + 1} - \hat{x}^k)}^2 \vspace{1.2ex}\\
	& - {~} \frac{(1 - \tau_k)}{2}\left[\rho_{k - 1} - (1 - \tau_k)\rho_k\right] \norms{Kx^k - r^k}^2.
\end{array}
\hspace{-2ex}
}
\end{lemma}

%
\beforesubsec
\subsection{\bf Parameter update and complete algorithm}\label{subsec:alg2_full}
\aftersubsec
As before, we can eliminate $r^k$ and $r^{k+1}$ in \eqref{eq:pd_scheme3} following the same lines as in Subsection \ref{subsec:alg1}, in order to transform \eqref{eq:pd_scheme3} into a primal-dual form. 
We also add an averaging sequence $\set{\bar{y}^{k}}$ as in Lemma \ref{le:key_est_scvx2} to obtain a dual convergence rate. 
Furthermore, as guided by Lemma \ref{le:key_est_scvx2}, we propose the following parameter update rules:
\myeq{eq:update_rule_scvx1}{
	\rho_k := \frac{\rho_0}{\tau_k^2}, \qquad \beta_k := \frac\Gamma{\rho_k \norms{K}^2}, \quad \text{and}\quad  \eta_k := (1 - \gamma)\rho_k,
}
where $\gamma \in (\frac12, 1)$ and $\Gamma := 2 - \frac1\gamma \in (0, 1)$ are given. For the choice of $\rho_0$ and the update of $\tau_k$, we provide two cases:
\myeq{eq:update_rule_scvx_init}{
	\textbf{Case 1:} \quad \rho_0 \in \left(0,\ \frac{\Gamma\mu_f}{2\norms{K}^2}\right]~ \text{and}~ \tau_{k + 1} := \frac{\tau_k}2 \left(\sqrt{\tau_k^2 + 4} - \tau_k\right),~ \text{where } \tau_0 := 1,{\!\!\!}
}
or
\myeq{eq:small_o_update_rule_scvx_init}{
	\textbf{Case 2:}~~ \rho_0 \in \left(0,\ \frac{c(c - 1)\Gamma\mu_f}{(2c - 1)\norms{K}^2}\right] \quad \text{and}\quad \tau_k := \frac{c}{k + c},\quad \text{where } c > 2~\text{is given}.
}

Now, we can describe our second first-order primal-dual algorithm as in Algorithm \ref{alg:A2}, and highlight the following features.

\begin{algorithm}[ht!]\caption{(New Primal-Dual Algorithm for \eqref{eq:com_cvx} and \eqref{eq:dual_prob}: Strongly Convex Case)}\label{alg:A2}
\normalsize
\begin{algorithmic}[1]
   \State{\bfseries Initialization:} Choose $y^0\in\R^n$, $x^0 \in\R^p$, and $\gamma \in (\frac12, 1)$. Set $\Gamma := 2 - \frac1\gamma$.
   \vspace{0.5ex}
   \State\label{A2_step:init_opt} Set $\rho_0$ and $\tau_0$ according to \eqref{eq:update_rule_scvx_init} or \eqref{eq:small_o_update_rule_scvx_init}.
   \vspace{0.5ex}
   \State\label{A2_step:dummy} Set $x^{-1} := \hat{x}^0 := x^0$ and $\tilde{y}^{-1} := \tilde{y}^0 := \bar{y}^0 := y^0$.
   \vspace{0.5ex}
   \State{\bfseries For $k := 0, 1, \cdots, k_{\max}$ do}
   		\vspace{0.5ex}
		\State\hspace{3ex}\label{A2_step:i1} Update $\rho_k := \frac{\rho_0}{\tau_k^2}$, ~ $\beta_k := \frac{\Gamma}{\rho_k \norm{K}^2}$, and $\eta_k := (1-\gamma)\rho_k$.
		\vspace{0.5ex}
		\State\hspace{3ex}\label{A2_step:tau} Update $\tau_{k + 1}$ according to \eqref{eq:update_rule_scvx_init} or \eqref{eq:small_o_update_rule_scvx_init}, consistent with the update in Step \ref{A2_step:init_opt}.
    	\vspace{0.5ex}   
   		\State\hspace{3ex}\label{A2_step:i2} Update the primal-dual step:
		\vspace{-1ex}
		\myeqn{
			\arraycolsep=0.27em
			\left\{\begin{array}{lcl}
				y^{k + 1}			& := & \prox_{\rho_k g^{\ast}} (\tilde{y}^k + \rho_k K\hat{x}^k),\vspace{1ex}\\
				\tilde{x}^{k + 1}	& := & \prox_{(\beta_k/\tau_k)f} (\tilde{x}^k - \frac{\beta_k}{\tau_k} K^\top y^{k + 1}),\vspace{1ex}\\
				x^{k + 1}			& := & \prox_{f/(\rho_k \norms{K}^2)}\left(\hat{x}^k - \frac1{\rho_k \norms{K}^2} K^{\top}y^{k + 1}\right),\vspace{1ex}\\
				\hat{x}^{k + 1}		& := & (1 - \tau_{k + 1})x^{k + 1} + \tau_{k + 1}\tilde{x}^{k + 1}.
			\end{array}\right.
			\vspace{-1ex}
		}
    	\State\hspace{3ex}\label{A2_step:i3} Update the intermediate dual step:
   	 	\vspace{-1ex}
		\myeqn{
			\arraycolsep=0.27em
			\begin{array}{lcl}
				\tilde{y}^{k+1} & := & \tilde{y}^k + \eta_kK\big[x^{k+1} - \hat{x}^k - (1-\tau_k)(x^k - \hat{x}^{k-1})\big]\vspace{1ex}\\
				& & + {~} (1 - \gamma)\Big[y^{k + 1} - \tilde{y}^k - \frac{\tau_{k - 1} (1 - \tau_k)}{\tau_k} (y^k - \tilde{y}^{k - 1})\Big].
			\end{array}
			\vspace{-1ex}
		}
     	\State\hspace{3ex}\label{A2_step:i4} Update the dual averaging step:	 $\bar{y}^{k+1} := (1-\tau_k)\bar{y}^k + \tau_ky^{k+1}$.
    	\vspace{0.5ex}   
	\State{\bfseries EndFor}
\end{algorithmic}
\end{algorithm}

\begin{itemize}
\item
Since we aim at obtaining non-ergodic convergence rate, Algorithm \ref{alg:A2} requires one additional $\prox_{f/(\rho_k \norms{K}^2)}(\cdot)$ compared to Algorithm~\ref{alg:A1}. 
We could replace this proximal step by an averaging step: $x^{k+1} := (1-\tau_k)x^k + \tau_k\tilde{x}^{k+1}$, but the convergence rate would no longer be non-ergodic in $\set{x^k}$.
	
\item
At each iteration, Algorithm~\ref{alg:A2} requires one $\prox_{\rho_kg^{\ast}}$, one $\prox_{\beta_kf}$, one\\ $\prox_{f/(\rho_k \norms{K}^2)}$, one $Kx$, and one $K^{\top}y$, 
which incur one more proximal operation of $f$ than in existing primal-dual methods. 
Again, the matrix-vector multiplication at Step~\ref{A2_step:i3} can be eliminated by storing vectors $Kx^k$ and $Kx^{k + 1}$.

\item 
Due to the symmetry between \eqref{eq:com_cvx} and \eqref{eq:dual_prob}, if $g^{\ast}$ is $\mu_{g^{\ast}}$-strongly convex with $\mu_{g^{\ast}} > 0$ (or equivalently, $g$ is $L_g$-smooth with $L_g := 1/\mu_{g^{*}}$), then we can apply Algorithm \ref{alg:A2} to the dual-primal pair instead of the primal-dual pair.
\end{itemize}

\beforesubsec
\subsection{\bf Convergence analysis}
\aftersubsec
We state the convergence of Algorithm~\ref{alg:A2} under Case 1, i.e., \eqref{eq:update_rule_scvx_init}, in the following theorem, whose proof is in Appendix \ref{apdx:th:big_O_convergence_rates_scvx}.

\begin{theorem}[$\BigO{1/k^2}$ convergence rates under Case 1]\label{th:convergence_guarantee_scvx1}
Let $\sets{(x^k, \bar{y}^k)}$ be generated by Algorithm~\ref{alg:A2} using the update \eqref{eq:update_rule_scvx1} and \eqref{eq:update_rule_scvx_init}, and $\widetilde{\Lc}$ be defined by \eqref{eq:min_max}. Then, under Assumptions~\ref{as:A0} and \ref{as:A2}, for any $x,\ x^0\in\R^p$ and $y,\ y^0\in\R^n$, we have
\myeq{eq:convergence_bound_scvx1}{
	\widetilde{\Lc}(x^k, y) - \widetilde{\Lc}(x, \bar{y}^k) \leq \frac{2}{(k+1)^2}\left[\frac{\rho_0 \norms{K}^2 \norms{x^0 - x}^2}{\Gamma} + \frac{\norms{y^0 - y}^2}{(1-\gamma)\rho_0}\right].
}
Moreover, the following statements hold:
\begin{itemize}
\item[\textnormal{(a)}] \textnormal{(Semi-ergodic convergence on the primal-dual gap).} 
The gap function $\Gc_{\Xc\times\Yc}$ defined by \eqref{eq:gap_func} satisfies the following bound for all $k \geq 1$:
\myeq{eq:primal_dual_bound_scvx1}{
	\Gc_{\Xc \times\Yc}(x^k, \bar{y}^k) \leq \frac2{(k +1)^2}  \sup_{(x,y)\in\Xc\times\Yc}\left[\frac{\rho_0\norms{K}^2\norms{x^0 -x}^2}{\Gamma} + \frac{\norms{y^0 - y}^2}{(1 -\gamma)\rho_0}\right]. {\!\!\!\!\!\!\!\!\!\!}
}
Hence, the primal-dual gap sequence $\set{\Gc_{\Xc\times\Yc}(x^k, \bar{y}^k)}$ converges to zero at $\BigO{1/k^2}{\!}$ semi-ergodic rate, i.e., non-ergodic in $x^k$ and ergodic in $\bar{y}^k$.

\item[\textnormal(b)] \textnormal{(Non-ergodic convergence on the primal objective).} 
If $g$ is $M_g$-Lipschitz continuous on $\dom{g}$ and $x^{\star}$ is an optimal solution of \eqref{eq:com_cvx}, then the primal objective residual based on the last-iterate sequence $\sets{x^k}$ satisfies:
\myeq{eq:primal_bound_scvx1}{
	0 \leq F(x^k) - F^{\star} \leq \frac{2}{(k+1)^2}\left[\frac{\rho_0\norms{K}^2\norms{x^0 - x^{\star}}^2}{\Gamma} + \frac{D^2_g}{(1-\gamma)\rho_0}\right],
}
where $D_g := \sup\set{\norms{y^0 - y} \mid \norm{y} \leq M_g}$. 
Hence, $\sets{F(x^k)}$ converges to the primal optimal value $F^{\star}$ of \eqref{eq:com_cvx} at $\BigO{1/k^2}$ rate in non-ergodic sense.

\item[\textnormal{(c)}] \textnormal{(Ergodic convergence on the dual objective).} 
If $f^{\ast}$ is $M_{f^{\ast}}$-Lipschitz continuous on $\dom{f^*}$ and $y^{\star}$ is an optimal solution of \eqref{eq:dual_prob}, then the dual objective residual based on the averaging sequence $\sets{\bar{y}^k}$ satisfies:
\myeq{eq:dual_bound_scvx1}{
	0 \leq G(\bar{y}^k) - G^{\star} \leq \frac{2}{(k+1)^2}\left[\frac{\rho_0\norms{K}^2 D_{f^{\ast}}^2}{\Gamma} + \frac{\norms{y^0 - y^{\star}}^2}{(1-\gamma)\rho_0}\right],
}
where $D_{f^{\ast}} := \sup\set{\norms{x^0 - x} \mid \norm{x} \leq M_{f^{\ast}}}$. 
Hence, $\sets{G(\bar{y}^k)}$ converges to the dual optimal value $G^{\star}$ of \eqref{eq:dual_prob} at $\BigO{1/k^2}$ rate in ergodic sense.
\end{itemize}
\end{theorem}

\begin{remark}[\textit{Optimal rate}]\label{rmk:optimal_rate_O(k^2)}
As shown in \cite[Theorem 2]{woodworth2016tight}, the  $\BigO{1/k^2}$ convergence rate of Algorithm~\ref{alg:A2} is optimal in the sense of Remark \ref{rmk:optimal_rate}.
Moreover, by Assumption~\ref{as:A2}, we can  show that $\set{\norms{x^k \!-\! x^{\star}}^2}$ converges to zero at $\BigO{1/k^2}$ rate.
\end{remark}

If we update the parameters using Case 2, i.e., \eqref{eq:small_o_update_rule_scvx_init}, then Algorithm \ref{alg:A2} achieves the same $\BigO{1/k^2}$ and  an empirically faster $\underline{o}\big(1/(k^2\sqrt{\log{k}})\big)$ rate on the primal objective residual, as shown in the following theorem, whose proof is given in Appendix~\ref{apdx:th:small_o_convergence_rate_scvx}.

\begin{theorem}[$\BigO{1/k^2}$ and $\underline{o}\big(1/(k^2\sqrt{\log{k}})\big)$ convergence rates under Case 2]\label{th:small_o_convergence_rate_scvx}
Let $\sets{x^k}$ be generated by Algorithm~\ref{alg:A2} using the update  \eqref{eq:update_rule_scvx1} and \eqref{eq:small_o_update_rule_scvx_init}. 
Let $\widetilde{\Lc}$ be defined by \eqref{eq:min_max} and $y^\star$ be an optimal solution of \eqref{eq:dual_prob}. Then, under Assumptions~\ref{as:A0} and~\ref{as:A2}, the following statements holds:
\myeq{eq:primal_bound_scvx1b}{
	0 \leq \widetilde{\Lc}(x^k, y^\star) - F^{\star} \leq \frac{R_0^2}{{(k + c - 1)}^2} ~~~\text{and}~~~ \liminf_{k \to \infty} \, k^2\log{(k)} \big[\widetilde{\Lc}(x^k, y^{\star}) - F^{\star} \big] = 0, {\!\!\!\!\!\!}
}
where $R_0^2 := (c \!-\! 1)\left[F(x^0) \!-\! F^\star\right] + \frac{c \!-\! 1}{2} \big[\frac{(c \!-\! 1)\rho_0 \norms{K}^2}\Gamma + c\mu_f\big]\norms{x^0 \!-\! x^\star}^2 + \frac{c^2}{2(1 - \gamma)\rho_0} \norms{y^0 - y^\star}^2$.

Moreover, if $g$ is $M_g$-Lipschitz continuous on $\dom{g}$, then the primal last-iterate sequence $\sets{x^k}$ satisfies
\myeq{eq:primal_bound_scvx1b_primal}{
	0 \leq F(x^k) - F^\star \leq \frac{R_1^2}{{(k + c - 1)}^2}\quad \text{and}\quad\quad \underset{k \to \infty}{\liminf} \, k^2\sqrt{\log{k}} \big[F(x^k) - F^\star\big] = 0,
}
where $R_1^2 := R_0^2 + \sqrt{2c^2/\rho_0}(\norms{y^\star} + M_g)R_0$. Hence, $\sets{F(x^k)}$ converges to the primal optimal value $F^\star$ of \eqref{eq:com_cvx} at both $\BigO{1/k^2}$ and $\underline{o}\big(1/(k^2 \sqrt{\log k})\big)$ convergence rates in non-ergodic sense.
\end{theorem}

\begin{remark}\label{re:comparison2}
The $\underline{o}\big(1/(k^2\sqrt{\log{k}})\big)$ convergence rate stated in Theorem~\ref{th:small_o_convergence_rate_scvx} is attained for sufficiently large $k$.
This does not conflict with the optimal upper bound stated in Remark~\ref{rmk:optimal_rate_O(k^2)}, where the number of iterations $k \leq \BigO{p}$ with $p$ being the dimension of the problem.
\end{remark}

\beforesubsec
\subsection{Application to constrained problems with semi-strongly convex objective}\label{subsec:constr_prob_scvx}
\aftersubsec
Consider the following constrained convex optimization problem:
\myeq{eq:constr_cvx2}{
	F^{\star} := \min_{x \in \R^p,\ w \in \R^q} \Big\{ F(x, w) := f(x) + \psi(w) \quad \text{s.t.}\quad Kx + Bw = b \Big\},
}
where $f : \R^p\to\Rext$ and $\psi : \R^q \to \Rext$ are proper, closed, and convex, $K\in\R^{n\times p}$, $B\in\R^{n\times q}$, and $b\in\R^n$. 
Different from \eqref{eq:constr_cvx}, we assume that:

\begin{assumption}\label{as:A3}
The first objective term $f$ is strongly convex with a modulus $\mu_f > 0$, but the second one $\psi$ is not necessarily strongly convex.{\!\!\!}
\end{assumption}
Note that problem \eqref{eq:constr_cvx2} is not necessarily strongly convex due to the separability of variables $x$ and $w$ in $f$ and $\psi$, respectively.

We modify Algorithm \ref{alg:A2} as follows to solve \eqref{eq:constr_cvx2}:
\myeq{eq:pd_scheme3_semistr}{
\arraycolsep=0.15em
\hspace{-1ex}
\left\{\begin{array}{lcl}
	w^{k + 1} & := & \argmin_{w}\set{ \psi(w) + \iprods{B^{\top}\tilde{y}^k, w} + \frac{\rho_k}{2}\norms{K\hat{x}^k + Bw - b}^2 + \frac{\nu_0}{2}\norms{w - \hat{w}^k}^2},\vspace{1ex}\\
	y^{k + 1} & := & \tilde{y}^k + \rho_k (K\hat{x}^k + Bw^{k + 1} - b),\vspace{1ex}\\		
	\tilde{x}^{k + 1} & := & \prox_{(\beta_k/\tau_k)f} \left(\tilde{x}^k - \frac{\beta_k}{\tau_k} K^\top y^{k + 1}\right),\vspace{0.5ex}\\
	x^{k + 1} & := & \prox_{f/(\rho_k \norms{K}^2)} \left(\hat{x}^k - \tfrac1{\rho_k\norms{K}^2} K^\top y^{k + 1}\right),\vspace{1ex}\\
	\hat{x}^{k + 1}	& :=	& (1 - \tau_{k + 1})x^{k + 1} + \tau_{k + 1} \tilde{x}^{k + 1},\vspace{1ex}\\
	\hat{w}^{k + 1}	& :=	& w^{k + 1} + \frac{\tau_{k + 1} (1 - \tau_k)}{\tau_k} (w^{k + 1} - w^k),\vspace{1ex}\\
	\tilde{y}^{k + 1} & := & \tilde{y}^k + \eta_k \left[(Kx^{k + 1} + Bw^{k + 1} - b) - (1 - \tau_k)(Kx^k + Bw^k - b)\right].
\end{array}\right.
\hspace{-3ex}
}
Here, the parameters $\tau_k,\ \rho_k,\ \beta_k$, and $\eta_k$ are updated as in Algorithm \ref{alg:A2}.

In \eqref{eq:pd_scheme3_semistr}, we combine Algorithm \ref{alg:A2} and an alternating strategy between $x$ and $w$, but we do not linearize the $w^{k+1}$-subproblem to avoid imposing strong convexity on $\psi$.
When necessary, we add a proximal term $\frac{\nu_0}{2}\norms{w-\hat{w}^k}^2$ to guarantee that the $w^{k+1}$-subproblem always has optimal solution.
Note that if $B$ is invertible, then \eqref{eq:constr_cvx2} reduces to \eqref{eq:com_cvx} with $g(\hat{K}x) := \psi(-B^{-1}(Kx - b))$, where $\hat{K} := -B^{-1} K$. 
In this case, we could apply accelerated proximal gradient methods in \cite{attouch2016rate,Beck2009} to the dual problem \eqref{eq:dual_prob} and using the strategy in \cite{necoara2014iteration,necoara2015complexity} to recover a primal approximate solution, but the optimal rate would no longer be non-ergodic. 
Our method is accelerated on the primal problem instead of the dual one as in \cite{necoara2014iteration,necoara2015complexity}.

Finally, we state the convergence of our new scheme \eqref{eq:pd_scheme3_semistr} to solve \eqref{eq:constr_cvx2} in the following corollary, whose proof is given in Appendix \ref{apdx:th:constr_alg_convergence_scvx}.

\begin{corollary}\label{th:constr_alg_convergence_scvx}
Let $\sets{(x^k,w^k, \bar{y}^k)}$ be generated by \eqref{eq:pd_scheme3_semistr} to solve \eqref{eq:constr_cvx2} and its dual problem under Assumptions \ref{as:A0} and \ref{as:A3}. 
Let $(x^\star, w^\star, y^\star)$ be a triple of primal-dual optimal solution. 
If we update the parameters as in \eqref{eq:update_rule_scvx1} and \eqref{eq:update_rule_scvx_init}, then{\!\!\!}
\myeq{eq:constr_alg_convergence2}{
	\vert F(x^k, w^k) - F^{\star} \vert \leq \dfrac{2R_0^2}{(k+1)^2}~~\text{and}~~ \norms{Kx^k + Bw^k - b} \leq \dfrac{2R_0^2}{(k+1)^2},
}
where $R_0^2 := \frac{\rho_0\norms{K}^2}{\Gamma}\norms{x^0 - x^{\star}}^2 + \nu_0\norms{w^0 - w^{\star}}^2 + \frac{1}{\rho_0(1-\gamma)}(2\norms{y^{\star}} + \norms{y^0} + 1)^2$.
Hence, the objective residual and the feasibility violation both converge to zero at $\BigO{1/k^2}$ rate in non-ergodic sense.

If we update the parameters as in \eqref{eq:update_rule_scvx1} and \eqref{eq:small_o_update_rule_scvx_init}, then the $\BigO{1/k^2}$ non-ergodic convergence rate bounds on $\vert F(x^k, w^k) - F^{\star} \vert $ and $\norms{Kx^k + Bw^k - b}$ still hold, and
\myeq{eq:constr_alg_convergence_faster2}{
	\displaystyle\liminf_{k\to\infty}\, k^2\sqrt{\log{k}} \left[\vert F(x^k, w^k) - F^{\star}\vert +  \norms{Kx^k + Bw^k - b} \right] = 0.
}
Hence, the objective residual and feasibility violation sequences both converge to zero at $\BigO{1/k^2}$ and $\underline{o}\big(1/(k^2 \sqrt{\log k})\big)$ non-ergodic convergence rate.
\end{corollary}

\begin{remark}\label{re:small_rate2}
To the best of our knowledge, Corollary~\ref{th:constr_alg_convergence_scvx} presents the first fast $\min\sets{\BigO{1/k^2},\ \underline{o}\big(1/(k^2\sqrt{\log{k}})\big)}$ convergence result for general constrained convex problem \eqref{eq:constr_cvx2} under the semi-strong convexity assumption, i.e., $f$ is strongly convex, but $\psi$ is non-strongly convex.
\end{remark}

\beforesec
\section{Numerical illustrations}\label{sec:num_experiments}
\aftersec
We verify the theoretical statements in this paper through two well-known examples. Our code is implemented in MATLAB (R2014b) running on a MacBook Pro with 2.7 GHz Intel Core i5 and 16GB memory, and available at \href{https://github.com/quoctd/PrimalDualCvxOpt}{{\color{blue}https://github.com/quoctd/PrimalDualCvxOpt}}. 

\beforesubsec
\subsection{\bf Ergodic vs. non-ergodic convergence rates}\label{subsec:LAD}
\aftersubsec
We consider the following nonsmooth composite convex minimization problem:
\myeq{eq:LAD}{
	F^{\star} := \min_{x\in\R^p}\Big\{ F(x) := f(x) + \norms{Kx - b}_1 \Big\},
}
where $K\in\R^{n\times p}$, $b\in\R^n$, and $f(x)$ is a regularizer. 
This problem fits template \eqref{eq:com_cvx} with $g(r) := \norms{r - b}_1$.
We compare Algorithm~\ref{alg:A1} with two well-established methods: Chambolle-Pock's method (CP) \cite{Chambolle2011} and ADMM \cite{Boyd2011}.
When $f$ is strongly convex, we compare Algorithm~\ref{alg:A2} with the strongly convex variant of CP (CP-scvx) \cite{Chambolle2011, chambolle2016ergodic}.

\beforepara
\paragraph{Case 1 $($General convex case$)$}\label{para:exp_lad_case1}

We choose $f(x) := \lambda\norm{x}_1$ in \eqref{eq:LAD} with a regularization parameter $\lambda = 0.05$. We generate the entries of $K$ from $\Nc(0, 1)$, and set $b := Kx^{\natural} + e$, where $x^{\natural}$ is an $s$-sparse vector, and $e$ is a sparse Gaussian noise with variance $\sigma^2 := 0.01$ and $10\%$ nonzero entries. 
The problem size is $(n, p) = (2000, 640)$.

We run two variants of Algorithm~\ref{alg:A1} with $c = 1$ and $c=2$. Since CP and ADMM both have $\BigO{1/k}$ convergence rate on only the \textit{ergodic} sequence of the relative objective residual $\frac{F(\bar{x}^k) - F^{\star}}{\max\set{1,\ \vert F^{\star}\vert}}$, we compare this sequence with the \textit{non-ergodic} sequence of Algorithm~\ref{alg:A1}, so that all algorithms under comparison have some theoretical guarantee. 
Here, $F^{\star}$ is computed by Mosek \cite{mosek} with the highest precision. 

For Algorithm~\ref{alg:A1}, we use $\rho_0 := 5\cdot\big(\frac{\gamma}{1-\gamma}\big)^{1/2} \cdot \frac{\norms{y^0 - y^{\star}}}{\norms{K}\norms{x^0 - x^{\star}}}$ with $\gamma := 0.999$ as guided by Theorems~\ref{th:convergence_guarantee1} and \ref{th:faster_convergence_rate}. For CP method, we choose its step-sizes $\rho := \rho_0$
and $\beta := \frac\gamma{\norms{K}^2\rho}$.
To be fair in our comparison, we also try step-sizes $0.1\rho_0$ and $10\rho_0$. 
For ADMM, we reformulate \eqref{eq:LAD} into the constrained problem \eqref{eq:constr_reform} by introducing $r := Kx - b$.
Similar to the CP method, we tune the penalty parameter for ADMM and find that three different values $\rho := 0.5\rho_0$, $10\rho_0$, and $30\rho_0$ represent the best range for $\rho$.

The relative objective residuals are plotted in Figure~\ref{fig:LAD_convergence} (left) for the non-ergodic (last-iterate) sequence of Algorithm~\ref{alg:A1} and for the ergodic (averaging) sequence of CP and ADMM. All algorithms achieve $\BigO{1/k}$ rate.
The ergodic sequences of CP and ADMM, while having theoretical convergence guarantees, are slower than ours. 

\beforepara
\paragraph{Case 2 (Strongly convex case)}
We choose $f(x) := \lambda\norms{x}_1 + \frac{\mu_f}2 \norms{x}^2$ in \eqref{eq:LAD} with $\lambda := 0.05$ and $\mu_f := 0.1$, and generate problem instances the same way as in \textit{Case 1} but with $50\%$ correlated columns in $K$.

Since $f$ is $\mu_f$-strongly convex, we test Algorithm \ref{alg:A2} on \eqref{eq:LAD}. 
If we use \eqref{eq:update_rule_scvx_init} to update parameters, then we choose $\gamma = 0.999$ and  $\rho_0^1 := \frac{\Gamma \mu_f}{2\norms{K}^2}$. We also run a variant with $\rho_0^{1+} := \tfrac{5\Gamma\mu_f}{2\norms{K}^2}$ since it leads to empirically better performance, suggesting that our analysis in Theorem \ref{th:convergence_guarantee_scvx1} may not be tight.
If we use \eqref{eq:small_o_update_rule_scvx_init} to update parameters, we choose $c := 4$, $\gamma = 0.75$, and $\rho_0^2 := \frac{c(c-1)\Gamma\mu_f}{(2c-1)\norms{K}^2}$. 
For comparison, we implement CP-scvx in \cite{Chambolle2011} with initial penalty parameter $\rho^\text{CP} := \frac1{\norms{K}}$ as suggested in convergence analysis in \cite{chambolle2016ergodic}. We also test its variants with $0.01\rho^\text{CP}$, $0.75\rho^\text{CP}$ and $5\rho^\text{CP}$.

The convergence behavior of this test is plotted in Figure~\ref{fig:LAD_convergence} (right). Both Algorithm~\ref{alg:A2} and CP-scvx show $\BigO{1/k^2}$ convergence rate as predicted by the theory. Algorithm~\ref{alg:A2} using the update \eqref{eq:small_o_update_rule_scvx_init}  with $c=4$ is the fastest. CP-scvx, on the other hand, is sensitive to the parameter choice, and even its best variant underperforms variants of Algorithm \ref{alg:A2} with parameters $\rho_0^{1+}$ and $\rho_0^2$.

\begin{figure}[htp!]
	\begin{center}
		\vspace{0ex}
		\includegraphics[width=1.0\textwidth]{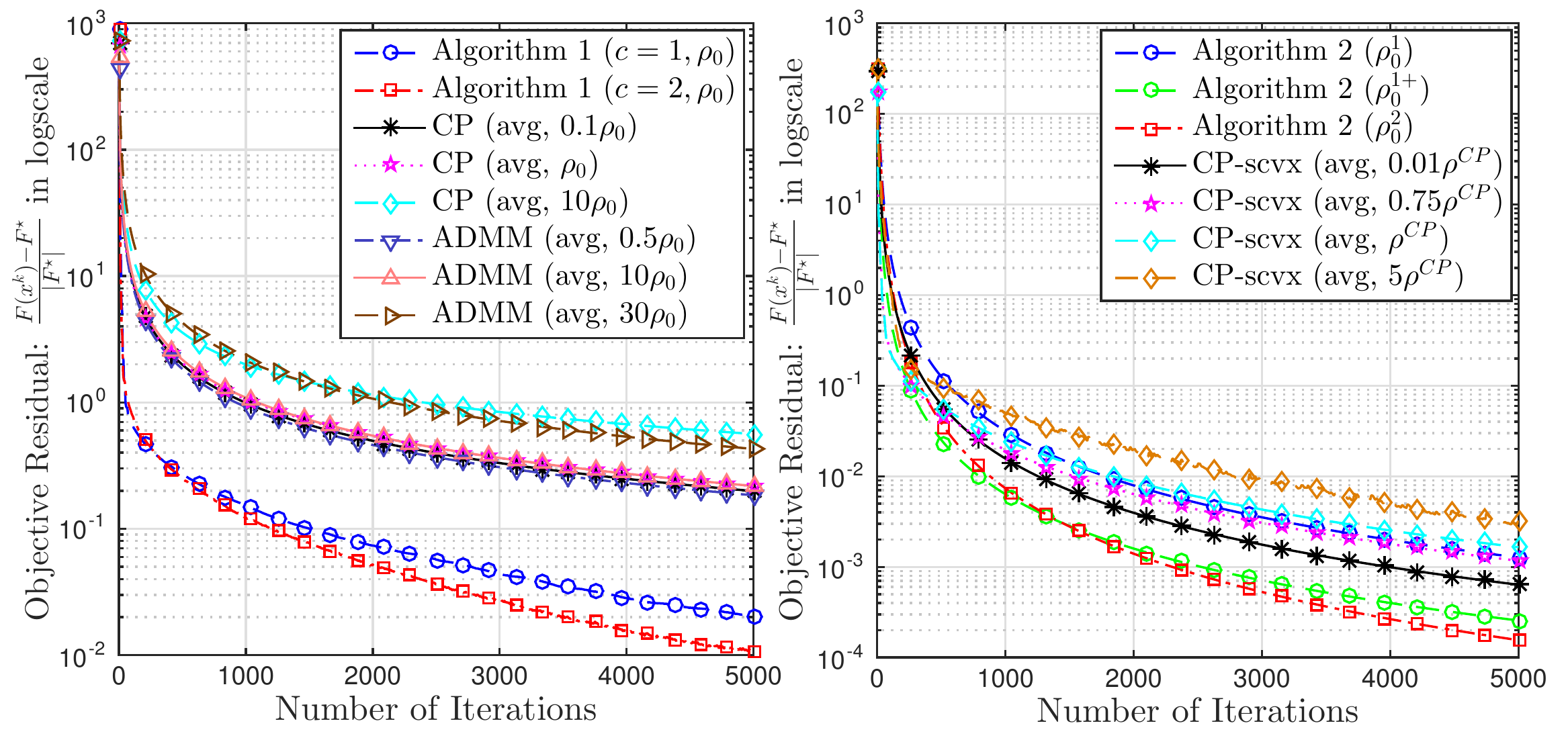} 
		\vspace{-3ex}
		\caption{Convergence behavior of algorithmic variants on \eqref{eq:LAD} with $K$ of size $(n, p) = (2000, 640)$. 
		Left: Case 1 $($general convex$)$ with $8$ variants; Right: Case 2 $($strongly convex$)$ with $7$ variants.}
		\label{fig:LAD_convergence}
		\vspace{0ex}
	\end{center}
\end{figure}

\beforesubsec
\subsection{\bf Primal-dual methods vs. smoothing techniques}\label{subsec:minmax_game}
\aftersubsec
Consider the following matrix min-max game problem studied, e.g., in \cite{Nesterov2005c}:
\myeq{eq:minmax_game}{
	F^{\star} := \min_{x\in\Delta_p} \Big\{ F(x) := \max_{y\in\Delta_n}\iprods{Kx, y} \Big\},
}
where $K \in \R^{n\times p},\ \Delta_p := \sets{ x\in\R^p_{+} \mid \sum_{j=1}^px_j = 1}$ and $\Delta_n := \set{ y\in\R^n_{+} \mid \sum_{i=1}^ny_i = 1}$ are two standard simplexes in $\R^p$ and $\R^n$, respectively. This problem can be cast into our template \eqref{eq:min_max} with $f(x) := \delta_{\Delta_p}(x)$ and $g^{\ast}(y) := \delta_{\Delta_n}(y)$, where $\Delta$ is the indicator function. In our experiment, the problem size is $(n, p) = (1000, 2000)$, and $K$ is 10\%-sparse with nonzero entries generated from Uniform$(-1, 1)$ distribution, then $K$ is normalized such that $\norms{K} = 1$.

We compare Algorithm~\ref{alg:A1} and Nesterov's smoothing technique in \cite{Nesterov2005c}. They both achieve the same theoretical $\BigO{1/k}$ convergence rate, but the performance of smoothing techniques depends on the choice of accuracy, as illustrated \cite{Nesterov2005c}.

For Algorithm~\ref{alg:A1}, we choose $\gamma := 0.5$ and $\rho_0 := \frac{1}{\norms{K}} = 1$ to balance the upper bound in Theorem~\ref{th:convergence_guarantee1}. We also update $\tau_k$ with $c := 1$ and $c := 2$ to obtain two variants.

For Nesterov's smoothing technique, since $\norms{K} = 1$, we use Euclidean distance to smooth $F(x) := \max_{y\in\Delta_n}\iprods{Kx,y}$ as $F_{\mu}(x) := \max_{y\in\Delta_n}\sets{\iprods{Kx,y} - \tfrac{\mu}{2}\norms{y - y_c}^2}$, which gives a better complexity bound than entropy proximity functions \cite[(4.11)]{Nesterov2005c}. Here $\mu > 0$ is the smoothness parameter and $y_c := (1/n,\cdots, 1/n)^{\top}$ is the center of $\Delta_n$. As suggested in \cite[(4.8)]{Nesterov2005c}, once the accuracy $\varepsilon > 0$ is fixed, we accordingly set the number of iterations $k_\text{max} := \frac{4\norms{K}}{\varepsilon} {\big[(1-\frac{1}{n})(1-\frac{1}{p})\big]}^{1/2}$ and the smoothness parameter $\mu := \frac{\varepsilon}{2(1-1/n)}$. We also run this algorithm with $5\mu$ and $\mu/5$ to observe its sensitivity to the choice of $\mu$.

We run Algorithm \ref{alg:A1} and Nesterov's smoothing method using the above configurations. 
If we test two cases with $\varepsilon_1 = 10^{-3}$ and $\varepsilon_2 = 10^{-4}$, then the corresponding numbers of iterations are $k_\text{max, 1} := 3,997$ and $k_\text{max, 2} := 39,970$, respectively. 
The duality gap $F(x^k) + G(\bar{y}^k)$ of this test is plotted in Figure~\ref{fig:duality_gap_c1}, where we observe that all algorithms indeed follow the $\BigO{1/k}$ convergence rate. However, the performance of Nesterov's smoothing method crucially depends on the choice of smoothness parameter $\mu$.
Algorithm~\ref{alg:A1} with $c=2$ outperforms all other methods in both cases.

\begin{figure}[htp!]
	\begin{center}
		\vspace{0ex}
		\includegraphics[width = 1.0\textwidth]{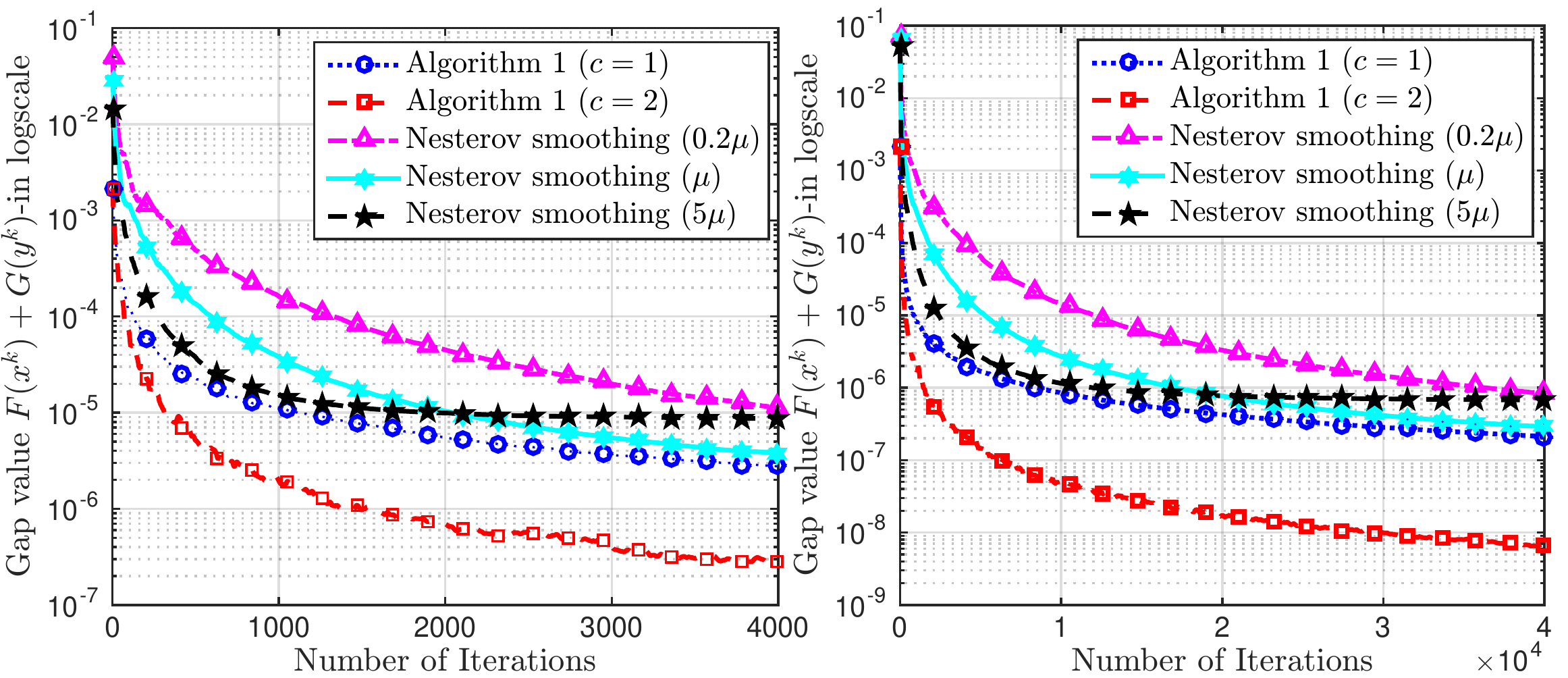}
		\vspace{-3ex}
		\caption{The convergence behavior of $5$ algorithmic variants on \eqref{eq:minmax_game} with a $10\%$-sparse matrix $K$ of size $(n, p) := (1000, 2000)$. 
		Left: $\varepsilon_1 = 10^{-3}$; Right: $\varepsilon_2 = 10^{-4}$.}
		\label{fig:duality_gap_c1}
		\vspace{0ex}
	\end{center}
\end{figure}

\beforesec
\section*{\small Acknowledgments}
\aftersec
{\small 
This paper is based upon work partially supported by the National Science Foundation (NSF), grant no. DMS-1619884, and the Office of Naval Research (ONR), grant No. N00014-20-1-2088.}

\appendix
\normalsize

\beforesec
\section{Some elementary results}
\aftersec
We recall some useful facts that will be used in the proof of our main results.

\begin{lemma}\label{lm:facts}
The following statements hold:
\begin{itemize}
\item[\textnormal{(a)}] Let $\phi_{\rho}(x, r, y) := \frac{\rho}{2}\norms{Kx - r}^2 + \iprods{y, Kx - r}$ be defined in \eqref{eq:aug_Lag} for $\rho > 0$.
Then, $\nabla_x\phi_{\rho}(x, r, y) = K^{\top} (y + \rho(Kx - r))$ and $\nabla_r \phi_{\rho}(x, r, y) = \rho(r - Kx) - y$. 
Furthermore, for any $x, x^\prime \in \R^p$ and $r, r^\prime, y \in \R^n$, we have
\myeqf{eq:phi_rho_properties}{
\arraycolsep=0.2em
\begin{array}{lcl}
	\phi_{\rho}(x^{\prime}, r^{\prime}, y) & = & \phi_{\rho}(x, r, y) + \iprods{\nabla_x\phi_{\rho}(x, r, y), x^{\prime} - x}\vspace{1.2ex}\\
	& & + {~} \iprods{\nabla_r\phi_{\rho}(x, r, y), r^{\prime} - r} + \frac{\rho}{2}\norms{K(x^{\prime} - x) - (r^{\prime} - r)}^2.
\end{array}
}
\item[\textnormal{(b)}] For any $u, v, w \in \R^p$ and $\alpha_1, \alpha_2 \in \R$, $\alpha_1 + \alpha_2 \neq 0$, it holds that
\myeqfn{
	\alpha_1 \norms{u - w}^2 + \alpha_2 \norms{v - w}^2 = (\alpha_1 + \alpha_2)\norms{w - \tfrac1{\alpha_1 + \alpha_2} (\alpha_1 u + \alpha_2 v)}^2 + \tfrac{\alpha_1 \alpha_2}{\alpha_1 + \alpha_2} \norms{u - v}^2.
}
\item[\textnormal{(c)}] If a nonnegative sequence $\set{u_k} \subseteq [0, +\infty)$ satisfies $\sum_{i = 0}^\infty u_k < +\infty$, then $\liminf_{k \to \infty} k\log{(k)} u_k = 0$.{\!\!\!} 
\item[\textnormal{(d)}] 
Let $\sets{u_k}$ and $\sets{v_k}$ be two nonnegative sequences in $\R$ and $\alpha_1, \alpha_2 \in \R_{++}$ be two positive constants.
Then, the following statements hold:
\begin{itemize}
\item[$\mathrm{(i)}$] If $\displaystyle\liminf_{k \to \infty}\big[ k\log{(k)} (u_k + \alpha_1 k v_k^2)\big] = 0$, then   $\displaystyle\liminf_{k \to \infty} \big[ k\sqrt{\log{k}} (u_k  +  \alpha_2 v_k) \big] = 0$.
\item[$\mathrm{(ii)}$] If $\displaystyle\liminf_{k \to \infty} \big[ k^2\log{(k)} (u_k  \ + \ \alpha_1 k^2 v_k^2) \big] =  0$, then  $\displaystyle\liminf_{k \to \infty} \big[ k^2\sqrt{\log{k}} (u_k  +  \alpha_2 v_k) \big] = 0$.
\end{itemize}
\end{itemize}
\end{lemma}

\begin{proof}
The statements (a) and (b) are trivial. 
We only prove parts (c) and (d). 

(c)~Since $u_k \geq 0$ and the $\liminf$ of a lower bounded sequence always exists, we set $\bar{u} := \liminf_{k \to \infty} k\log{(k)}u_k \geq 0$.{\!} 
Suppose that $\bar{u} > 0$. 
Then, by definition, for any $\varepsilon > 0$ such that $\bar{u} - \varepsilon > 0$, there exists an integer $k_\varepsilon > 0$ such that for any $k \geq k_\varepsilon$, we have $k\log{(k)}u_k \geq \bar{u} - \varepsilon$.
This leads to
\myeqn{
+\infty > \sum_{k = 0}^\infty u_k \geq \sum_{k = k_\varepsilon}^\infty u_k \geq \sum_{k = k_\varepsilon}^\infty \frac{\bar{u} - \varepsilon}{k\log{k}} = (\bar{u} - \varepsilon)\sum_{k = k_\varepsilon}^\infty \frac{1}{k\log{k}} = +\infty,
}
which is a contradiction.
Hence, we must have $\bar{u} = 0$.

(d)~Since $\liminf_{k \to \infty} k\log{(k)}(u_k + \alpha_1 k v_k^2) = 0$, there exists a convergent subsequence $\sets{k_j\log(k_j) (u_{k_j} + \alpha_1 k_j v_{k_j}^2)}_{j \geq 0}$ converging to $0$, i.e., for any $\varepsilon > 0$, there exists $j_0\geq 0$ such that for all $j \geq j_0$, we have $k_j\log(k_j) (u_{k_j} + \alpha_1 k_j v_{k_j}^2) < \min\set{\frac\varepsilon2,\  \frac{\alpha_1\varepsilon^2}{4\alpha_2^2} }$.
This inequality implies that
\begin{equation*}
\begin{array}{ll}
& k_j\sqrt{\log(k_j)} u_{k_j} < \frac{\varepsilon}{2} \vspace{1ex}\\
\text{and} & \alpha_2 k_j\sqrt{\log(k_j)} v_{k_j} = \frac{\alpha_2}{\sqrt{\alpha_1}} \sqrt{\alpha_1 k_j^2\log(k_j) v_{k_j}^2} < \frac{\alpha_2}{\sqrt{\alpha_1}} \sqrt{ \frac{\alpha_1\varepsilon^2}{4\alpha_2^2}} = \frac\varepsilon2.
\end{array}
{\!\!}\end{equation*}
Combining both inequalities, we can show that $k_j\sqrt{\log(k_j)}(u_{k_j} + \alpha_1 k_j v_{k_j}) < \tfrac\varepsilon2 + \tfrac\varepsilon2 = \varepsilon$, which proves part (i) of (d). 
Part (ii) of (d) can be proved analogously.
\end{proof}

\beforesec
\section{Technical proofs in Section~\ref{sec:new_pd_algs}: General convex case}\label{apdx:proofs1}
\aftersec
This appendix provides the full proof of technical results in Section~\ref{sec:new_pd_algs}.

\beforesubsec
\subsection{\textbf{The proof of Lemma~\ref{le:key_bound1}: One-iteration analysis}}\label{apdx:le:key_bound1}
\aftersubsec
First, we write down the optimality conditions of $x^{k+1}$ and $r^{k+1}$ in \eqref{eq:pd_scheme1} as follows:
\myeqf{eq:lm11_est1}{
\arraycolsep=0.3em
{\!\!\!}\left\{\begin{array}{lcl}
	0 & \in & \partial{g}(r^{k+1}) + \rho_k(r^{k+1} - K\hat{x}^k) - \tilde{y}^k \equiv \partial{g}(r^{k+1}) + \nabla_r\phi_{\rho_k}(\hat{x}^k, r^{k+1}, \tilde{y}^k),\vspace{1.2ex}\\
	0 & \in & \partial{f}(x^{k+1}) + \nabla_x\phi_{\rho_k}(\hat{x}^k, r^{k+1}, \tilde{y}^k) + \frac{1}{\beta_k}(x^{k+1} - \hat{x}^k).
\end{array}\right.{\!\!\!}
}
By convexity of $f$ and $g$, and the above optimality conditions, we can derive
\myeqf{eq:lm11_est1b}{
\hspace{-1ex}\left\{\arraycolsep=0.1em
\begin{array}{lcl}
	g(r^{k+1}) & \leq & g(r)\!+\!\iprods{\nabla{g}(r^{k+1}), r^{k+1}\!-\!r}\!\stackrel{\eqref{eq:lm11_est1}}=\!g(r)\!+\!\iprods{\nabla_r{\phi_{\rho_k}}(\hat{x}^k, r^{k+1}, \tilde{y}^k), r\!-\!r^{k+1}},\vspace{1.25ex}\\
	f(x^{k+1}) & \leq & f(x) + \iprods{\nabla{f}(x^{k+1}), x^{k+1} - x} \vspace{1ex}\\
	& \stackrel{\eqref{eq:lm11_est1}}= & f(x) + \iprods{\nabla_x{\phi_{\rho_k}}(\hat{x}^k, r^{k+1}, \tilde{y}^k), x - x^{k+1}} + \tfrac{1}{\beta_k}\iprods{x^{k+1} - \hat{x}^k,  x - x^{k+1}},
\end{array}{\!\!}\right.
\hspace{-1.5ex}
}
where $\nabla{f}(x^{k+1}) \in \partial{f}(x^{k+1})$ and $\nabla{g}(r^{k+1}) \in  \partial{g}(r^{k+1})$. 

\noindent Next, using Lemma \ref{lm:facts}(a) twice, we can derive
\myeqf{eq:lml1_melodyref1_prep}{
\arraycolsep=0.3em
{\!\!\!}\left\{\begin{array}{l}
	\phi_{\rho_k} (x^{k + 1}, r^{k + 1}, \tilde{y}^k) = \phi_{\rho_k} (\hat{x}^k, r^{k + 1}, \tilde{y}^k) + \iprods{\nabla_x \phi_{\rho_k} (\hat{x}^k, r^{k + 1}, \tilde{y}^k), x^{k + 1} - \hat{x}^k}\vspace{1ex}\\
	\phantom{\phi_{\rho_k}}  {~~~~~~} + \frac{\rho_k}2 \norms{K(x^{k + 1} - \hat{x}^k)}^2,\vspace{1.25ex}\\
	\phi_{\rho_k} (x, r, \tilde{y}^k) = \phi_{\rho_k} (\hat{x}^k, r^{k + 1}, \tilde{y}^k) + \iprods{\nabla_x \phi_{\rho_k} (\hat{x}^k, r^{k + 1}, \tilde{y}^k), x - \hat{x}^k}\vspace{1ex}\\
	\phantom{\phi_{\rho_k}}{~~~~~~} + \iprods{\nabla_r \phi_{\rho_k} (\hat{x}^k, r^{k + 1}, \tilde{y}^k), r - r^{k + 1}} + \frac{\rho_k}2 \norms{K(x - \hat{x}^k) - (r - r^{k + 1})}^2.{\!}
\end{array}\right.
}
Combining the two expressions in \eqref{eq:lml1_melodyref1_prep}, we get
\myeqf{eq:lm11_melodyref1}{
\hspace{-1ex}\arraycolsep=0.3em
\begin{array}{lcl}
	\phi_{\rho_k} (x^{k + 1}, r^{k + 1}, \tilde{y}^k) & = & \phi_{\rho_k} (x, r, \tilde{y}^k) + \iprods{\nabla_x \phi_{\rho_k} (\hat{x}^k, r^{k + 1}, \tilde{y}^k), x^{k + 1} - x}\vspace{1.25ex}\\
	& & + {~} \iprods{\nabla_r \phi_{\rho_k} (\hat{x}^k, r^{k + 1}, \tilde{y}^k), r^{k + 1} - r} + \frac{\rho_k}2 \norms{K(x^{k + 1} - \hat{x}^k)}^2\vspace{1.25ex}\\
	& & - {~} \frac{\rho_k}2 \norms{K(x - \hat{x}^k) - (r - r^{k + 1})}^2.
\end{array}
\hspace{-2ex}
}
Summing up \eqref{eq:lm11_est1b} and \eqref{eq:lm11_melodyref1}  and using \eqref{eq:aug_Lag}, we arrive at
\myeqf{eq:lm11_est1b2}{
\arraycolsep=0.3em
{\!\!\!\!\!}\begin{array}{lcl}
	\Lc_{\rho_k}(x^{k+1}, r^{k+1}, \tilde{y}^k) & \leq & \Lc_{\rho_k}(x, r, \tilde{y}^k) + \frac{1}{\beta_k}\iprods{x^{k+1} - \hat{x}^k, x - \hat{x}^k} - \frac{1}{\beta_k}\norms{x^{k+1} - \hat{x}^k}^2 \vspace{1ex}\\
	& & + {~} \frac{\rho_k}{2}\norms{K(x^{k+1} - \hat{x}^k)}^2  - \frac{\rho_k}{2}\norms{K(x - \hat{x}^k) - (r - r^{k+1})}^2.
\end{array}{\!\!\!\!\!}
}
Since \eqref{eq:lm11_est1b2} holds for any $x \in \R^p$ and $r \in \R^n$, we can substitute $(x, r) := (x^k, r^k)$ to obtain
\myeqf{eq:lm11_est3}{
\arraycolsep=0.2em
{\!\!\!}\begin{array}{rl}
	\Lc_{\rho_k} & (x^{k+1}, r^{k+1}, \tilde{y}^k) \leq \Lc_{\rho_k}(x^k, r^k, \tilde{y}^k) + \frac{1}{\beta_k}\iprods{x^{k+1} - \hat{x}^k, x^k - \hat{x}^k} \vspace{1ex}\\
	& - {~} \frac{1}{\beta_k}\norms{x^{k+1} - \hat{x}^k}^2 + \frac{\rho_k}{2}\norms{K(x^{k\!+\!1} - \hat{x}^k)}^2 - \frac{\rho_k}{2}\norms{K(x^k - \hat{x}^k) - (r^k - r^{k\!+\!1})}^2.
\end{array}{\!\!\!}
}
Multiplying \eqref{eq:lm11_est1b2} by $\tau_k$ and \eqref{eq:lm11_est3} by $1-\tau_k$, summing up the results, then utilizing $\hat{x}^k = (1-\tau_k)x^k + \tau_k\tilde{x}^k$ from \eqref{eq:pd_scheme1}, we get
\myeqf{eq:lm11_est4}{
\arraycolsep=0.3em
{\!\!\!}\begin{array}{rl}
	\Lc_{\rho_k} & (x^{k+1}, r^{k+1}, \tilde{y}^k) \leq (1-\tau_k)\Lc_{\rho_k}(x^k, r^k, \tilde{y}^k) + \tau_k\Lc_{\rho_k}(x, r, \tilde{y}^k)\vspace{1ex}\\
	& + {~} \frac{\tau_k}{\beta_k}\iprods{x^{k+1} - \hat{x}^k, x - \tilde{x}^k} - \frac{1}{\beta_k}\norms{x^{k+1} - \hat{x}^k}^2 + \frac{\rho_k}{2}\norms{K(x^{k+1} - \hat{x}^k)}^2\vspace{1ex}\\
	&  - {~} \frac{\tau_k \rho_k}{2}\norms{K(\hat{x}^k - x) - (r^{k+1} - r)}^2 - \frac{(1 - \tau_k)\rho_k}{2}\norms{K(x^k - \hat{x}^k) - (r^k - r^{k+1})}^2.
\end{array}{\!\!\!}
}
By the definition of $\tilde{y}^{k+1}$ from \eqref{eq:pd_scheme1}, we have
\myeqf{eq:lm11_est7}{
\arraycolsep=0.3em
{\!\!\!}\begin{array}{rl}
	\Lc_{\rho_k}\!& (x^{k+1}, r^{k+1},  y) - (1-\tau_k)\Lc_{\rho_{k}}(x^k, r^k, y) = \Lc_{\rho_k}(x^{k+1}, r^{k+1}, \tilde{y}^k) \vspace{1.25ex}\\
	& - {~} (1-\tau_k)\Lc_{\rho_{k}}(x^k, r^k, \tilde{y}^k) + \frac{1}{ 2\eta_k}\left[\norms{\tilde{y}^k - y}^2 - \norms{\tilde{y}^{k + 1} - y}^2 + \norms{\tilde{y}^{k+1} - \tilde{y}^k}^2\right].
\end{array}{\!\!\!}
}
Substituting the expression \eqref{eq:lm11_est7} into \eqref{eq:lm11_est4}, we get
\myeqf{eq:T2_proofmelody1}{
\hspace{-1ex}\arraycolsep=0.2em
{\!\!\!}\begin{array}{rl}
	\Lc_{\rho_k} (x^{k + 1}, & r^{k + 1}, y) \leq (1 - \tau_k)\Lc_{\rho_k} (x^k, r^k, y)\hfill (=: \Tc_1)\vspace{1ex}\\
	& + {~} \tau_k \Lc_{\rho_k} (x, r, \tilde{y}^k) - \tfrac{\tau_k \rho_k}2 \norms{K\hat{x}^k - r^{k + 1} - (Kx - r)}^2\hfill (=: \Tc_2)\vspace{1ex}\\
	& + {~} \tfrac{\tau_k}{\beta_k} \iprods{x^{k + 1} - \hat{x}^k, x - \tilde{x}^k} - \tfrac1{2\beta_k} \norms{x^{k + 1} - \hat{x}^k}^2\hfill (=: \Tc_3)\vspace{1ex}\\
	& - {~} \tfrac1{2\beta_k} \norms{x^{k + 1} - \hat{x}^k}^2 + \tfrac1{2\eta_k} \left[\norms{\tilde{y}^k - y}^2 - \norms{\tilde{y}^{k + 1} - y}^2 + \norms{\tilde{y}^{k + 1} - \tilde{y}^k}^2\right]\vspace{1ex}\\
	& + {~} \tfrac{\rho_k}2 \norms{K(x^{k + 1} - \hat{x}^k)}^2 - \tfrac{(1 - \tau_k)\rho_k}2 \norms{K(x^k - \hat{x}^k) - (r^k - r^{k + 1})}^2.
\end{array}{\!\!}
}
Now, we  estimate three terms $\Tc_1$, $\Tc_2$, and $\Tc_3$ in \eqref{eq:T2_proofmelody1} as follows. 
First, we have
\myeqf{eq:Lemma2_Tc1}{
	\Tc_1 = (1 - \tau_k)\left[\Lc_{\rho_{k - 1}} (x^k, r^k, y) + \frac{(\rho_k - \rho_{k - 1})}{2} \norms{Kx^k - r^k}^2\right].
}
By the definitions of $\bar{y}^{k + 1}$ in Lemma~\ref{le:key_bound1} and of $\Lc$, we can show that
\myeqf{eq:T2_proof5}{
	\Tc_2 = \Lc(x, r, \bar{y}^{k+1}) - (1-\tau_k)\Lc(x, r, \bar{y}^k) -\frac{\tau_k\rho_k}{2}\norms{K\hat{x}^k - r^{k+1}}^2.
}
Moreover, by the update $\tilde{x}^{k+1}  := \tilde{x}^k  + \frac{1}{\tau_k}(x^{k+1} - \hat{x}^k)$ in \eqref{eq:pd_scheme1}, we can further derive
\myeqfn{
	\Tc_3 = \frac{\tau_k^2}{2\beta_k}\left[ \norms{\tilde{x}^k - x}^2 - \norms{\tilde{x}^{k+1} - x}^2\right].
}
Substituting these three terms $\Tc_1$, $\Tc_2$, and $\Tc_3$ back into \eqref{eq:T2_proofmelody1}, we obtain
\myeqfn{
\arraycolsep=0.3em
{\!\!\!}\begin{array}{lcl}
	\Lc_{\rho_k} (x^{k + 1}, r^{k + 1}, y) & \leq & (1 - \tau_k)\Lc_{\rho_{k - 1}} (x^k, r^k, y) + \frac{(1 - \tau_k)(\rho_k - \rho_{k - 1})}2 \norms{Kx^k - r^k}^2 \vspace{1ex}\\
	& & + {~} \Lc(x, r, \bar{y}^{k + 1}) - (1 - \tau_k)\Lc(x, r, \bar{y}^k) - \frac{\tau_k \rho_k}2 \norms{K\hat{x}^k - r^{k + 1}}^2 \vspace{1ex}\\
	& & + {~} \frac{\tau_k^2}{2\beta_k} \left[ \norms{\tilde{x}^k - x}^2 - \norms{\tilde{x}^{k + 1} - x}^2\right]  - \frac{1}{2\beta_k} \norms{x^{k + 1} - \hat{x}^k}^2 \vspace{1ex}\\
	& & + {~} \frac{1}{2\eta_k} \left[ \norms{\tilde{y}^k - y}^2 - \norms{\tilde{y}^{k + 1} - y}^2\right] + \frac{1}{2\eta_k} \norms{\tilde{y}^{k + 1} - \tilde{y}^k}^2 \vspace{1ex}\\
	& & + {~} \frac{\rho_k}2 \norms{K(x^{k + 1} - \hat{x}^k)}^2 - \frac{(1 - \tau_k)\rho_k}{2} \norms{K(x^k - \hat{x}^k) - (r^k - r^{k + 1})}^2,
\end{array}{\!\!\!}
}
which, after rearrangement, becomes
\myeqf{eq:Lemma2_estimate_melody2}{
\arraycolsep=0.2em
{\!\!\!\!\!}\begin{array}{rl}
	\Lc_{\rho_k} & (x^{k + 1}, r^{k + 1}, y) - \Lc(x, r, \bar{y}^{k + 1}) \leq (1 - \tau_k)\left[\Lc_{\rho_{k - 1}} (x^k, r^k, y) - \Lc(x, r, \bar{y}^k)\right]\vspace{1ex}\\
	& + {~} \frac{\tau_k^2}{2\beta_k} \left[ \norms{\tilde{x}^k - x}^2 - \norms{\tilde{x}^{k + 1} - x}^2\right]  - \frac{1}{2\beta_k} \norms{x^{k + 1} - \hat{x}^k}^2 + \frac{\rho_k}2 \norms{K(x^{k + 1} - \hat{x}^k)}^2\vspace{1ex}\\
	& + {~} \frac{1}{2\eta_k} \left[ \norms{\tilde{y}^k - y}^2 - \norms{\tilde{y}^{k + 1} - y}^2\right] + \frac{1}{2\eta_k} \norms{\tilde{y}^{k + 1} - \tilde{y}^k}^2  + \Tc_4,
\end{array}{\!\!\!\!\!}
}
where
\myeqf{eq:lm11_est8}{
\arraycolsep=0.3em
\hspace{-1ex}\begin{array}{lrl}
	\Tc_4 & := & \frac{(1-\tau_k)}{2}\left( \rho_k - \rho_{k-1} \right)\norms{Kx^k - r^k}^2 - \frac{\tau_k \rho_k}{2}\norms{K\hat{x}^k - r^{k+1}}^2\vspace{1ex}\\
	& & - {~} \frac{(1-\tau_k)\rho_k}{2}\norms{K(x^k - \hat{x}^k) - (r^k - r^{k+1})}^2\vspace{1ex}\\
	& = & - \frac{\rho_k}{2}\norms{(K\hat{x}^k\!- r^{k+1}) - (1-\tau_k)(Kx^k- r^k)}^2 \vspace{1ex}\\
	& & - {~} \frac{(1-\tau_k)}{2}\left[\rho_{k-1}-(1-\tau_k)\rho_k\right] \norms{Kx^k\!-\!r^k}^2.
\end{array}
\hspace{-2ex}
}
Using Lemma \ref{lm:facts}(b) with $u := Kx^{k + 1} - r^{k + 1}$, $v := K\hat{x}^k - r^{k + 1}$, $w := (1 - \tau_k)(Kx^k - r^k)$, $\alpha_1 := \eta_k/2$ and $\alpha_2 := -\rho_k/2$, and $\tilde{y}^{k + 1}$ from \eqref{eq:pd_scheme1} with $\rho_k > \eta_k$, we can show that
\myeqf{eq:lm11_est9}{
\small
\frac{1}{2\eta_k}\norms{\tilde{y}^{k + 1} {\!} - \tilde{y}^k}^2 - \frac{\rho_k}{2}\norms{(K\hat{x}^k - r^{k+1}) - (1-\tau_k)(Kx^k - r^k)}^2 \leq \frac{\rho_k\eta_k}{2(\rho_k - \eta_k)}\norms{K(x^{k+ 1}{\!\!} - \hat{x}^k)}^2.{\!\!\!\!\!\!\!\!\!}
}
Substituting this estimate and \eqref{eq:lm11_est8} into \eqref{eq:Lemma2_estimate_melody2}, we finally arrive at \eqref{eq:key_bound1}.
\proofbox

\beforesubsec
\subsection{\bf The proof of Theorem~\ref{th:convergence_guarantee1}: $\BigO{1/k}$ convergence rates when $c = 1$}\label{apdx:th:convergence_guarantee1}
\aftersubsec
Using the parameter update rule \eqref{eq:update_rule1} with $c := 1$, we can easily verify that
\myeqfn{
	\frac{\tau_k^2}{\beta_k} = \frac{(1 - \tau_k)\tau_{k - 1}^2}{\beta_{k - 1}},~~ \frac{1}{\eta_k} = \frac{1 - \tau_k}{\eta_{k - 1}}, ~~ \frac{1}{\beta_k} - \frac{\rho_k^2 \norms{K}^2}{\rho_k - \eta_k} = 0, 
	~~\text{and}~~  \rho_{k - 1} - (1 - \tau_k)\rho_k = 0.
}
Applying these conditions to \eqref{eq:key_bound1} of Lemma \ref{le:key_bound1}, we can simplify it as
\myeqfn{
\hspace{-2ex}\arraycolsep=0.3em
\begin{array}{rl}
	\Lc_{\rho_k} & (x^{k + 1}, r^{k + 1}, y) - \Lc(x, r, \bar{y}^{k + 1}) + \frac{\tau_k^2}{2\beta_k} \norms{\tilde{x}^{k + 1} - x}^2 + \frac{1}{2\eta_k} \norms{\tilde{y}^{k + 1} - y}^2\vspace{1ex}\\
	&{~~} \leq  (1 - \tau_k)\left[\Lc_{\rho_{k - 1}} (x^k, r^k, y) - \Lc(x, r, \bar{y}^k) + \frac{\tau_{k - 1}^2}{2\beta_{k - 1}} \norms{\tilde{x}^k - x}^2 + \frac{1}{2\eta_{k - 1}} \norms{\tilde{y}^k - y}^2\right].
\end{array}
\hspace{-2ex}
}
By induction, this inequality implies
\myeqfn{
\arraycolsep=0.3em
\begin{array}{rl}
	\Lc_{\rho_{k - 1}} & (x^k, r^k, y) - \Lc(x, r, \bar{y}^k) \leq \left[\prod\nolimits_{i = 1}^{k - 1} {\!\!}(1 - \tau_i)\right] \times\vspace{1ex}\\
	& \left[(1 - \tau_0)\left(\Lc_{\rho_{-1}} (x^0, r^0, y) - \Lc(x, r, \bar{y}^0)\right) + \tfrac{\tau_0^2}{2\beta_0} \norms{\tilde{x}^0 - x}^2 + \tfrac{1}{2\eta_0} \norms{\tilde{y}^0 - y}^2\right].
\end{array}
}
If $c = 1$, then $\tau_0 = 1$,\, $\beta_0 = \gamma/\left(\norms{K}^2 \rho_0\right)$, and $\eta_0 = (1 - \gamma)\rho_0$. 
We also have $\Pi_{i = 1}^{k - 1} (1 - \tau_i) = \frac{1}{k}$, $\tilde{x}^0 = x^0$, and $\tilde{y}^0 = y^0$. 
Thus the last estimate can be simplified as
\myeqf{eq:thm1_proof0_telescope_b}{
	 \Lc_{\rho_{k - 1}} (x^k, r^k, y) - \Lc(x, r, \bar{y}^k) \leq  \frac{1}{2k} \left[ \frac{\rho_0\norms{K}^2\norms{x^0 - x}^2 }{\gamma}  + \frac{\norms{y^0 - y}^2}{(1-\gamma)\rho_0}\right]. 
}
Now, let pick any $\bar{r}^k \in \partial g^{\ast} (\bar{y}^k)$. 
Then, by \eqref{eq:lag_func12}, we easily get
\myeqf{eq:thm1_proof0_L}{
	\widetilde{\Lc}(x^k, y) - \widetilde{\Lc}(x, \bar{y}^k) \leq \Lc(x^k, r^k, y) - \Lc(x, \bar{r}^k, \bar{y}^k).
}
Combining \eqref{eq:thm1_proof0_telescope_b}, \eqref{eq:thm1_proof0_L}, and $\Lc(x^k, r^k, y) \leq \Lc_{\rho_{k - 1}} (x^k, r^k, y)$, we finally get \eqref{eq:convergence_bound1}.

(a) The estimate~\eqref{eq:primal_dual_bound1} directly follows from \eqref{eq:convergence_bound1} and the definition of $\Gc_{\Xc\times\Yc}$ in \eqref{eq:gap_func}.

(b)~By $M_g$-Lipschitz continuity of $g$, we have
\myeqf{eq:th31_proof2}{
\arraycolsep=0.3em
{\!\!\!}\begin{array}{lcl}
	F(x^k) - F^{\star} & \leq & f(x^k) + g(r^k) + M_g\norms{Kx^k - r^k} - F^{\star} \vspace{1ex}\\
	& = & f(x^k) + g(r^k) + \iprods{\breve{y}^k, Kx^k - r^k} - F^{\star}~~~\text{with}~\breve{y}^k := \frac{M_g(Kx^k - r^k)}{\norms{Kx^k - r^k}} \vspace{1ex}\\
	& \leq & \Lc_{\rho_{k - 1}}(x^k, r^k, \breve{y}^k) - \mathcal{L}(x^{\star}, r^{\star}, \bar{y}^k),
\end{array}{\!\!\!}
}
where we have used $\Lc(x^k, r^k, \breve{y}^k) = f(x^k) + g(r^k) + \iprods{Kx^k - r^k, \breve{y}^k} \leq \Lc_{\rho_{k-1}}(x^k, r^k, \breve{y}^k)$ and $F^{\star} = \Lc(x^{\star}, r^{\star}, \bar{y}^k)$ in the last inequality. 
Substituting $(x, r, y) := (x^{\star}, r^{\star}, \breve{y}^k)$ into \eqref{eq:thm1_proof0_telescope_b}, we have
\myeqfn{
	0 \leq F(x^k) - F^{\star} \leq \frac1{2k} \left[\frac{\rho_0 {\norms{K}}^2{\norms{x^0 - x^{\star}}}^2}{\gamma}  + \frac{ {\norms{y^0 - \breve{y}^k}^2}}{(1 - \gamma)\rho_0} \right].
}
By the definition of $\breve{y}^k$ in \eqref{eq:th31_proof2}, we have $\norms{y^0 - \breve{y}^k} \leq \sup_{y}\set{ \norms{y^0 - y} \mid \norm{y} \leq M_g} =: D_g$.
Using this estimate into the last inequality, we obtain \eqref{eq:primal_bound1}.

(c)~For any $x \in \R^p$ and $r\in\R^n$, we have
\myeqfn{
\arraycolsep=0.3em
\begin{array}{lcl}
	\Lc(x, r, \bar{y}^k) & = & f(x) - \iprods{-K^{\top}\bar{y}^k, x}  + g(r) - \iprods{\bar{y}^k, r} \vspace{1ex}\\
	& \geq & -\sup_{x}\big\{ \iprods{-K^{\top}\bar{y}^k, x} - f(x) \big\} - \sup_r\big\{\iprods{\bar{y}^k, r} - g(r) \big\} \vspace{1ex}\\
	& = & -f^{\ast}(-K^{\top}\bar{y}^k) - g^{\ast}(\bar{y}^k) = -G(\bar{y}^k).
\end{array}
}
Let $\breve{x}^k \in \partial {f^{\ast}}(-K^{\top}\bar{y}^k)$ and $\breve{r}^k \in \partial {g^{\ast}}(\bar{y}^k)$.
Then, it is clear that the above inequality holds as equality with $(x, r) := (\breve{x}^k, \breve{r}^k)$. 
We further have
\myeqfn{
	G(\bar{y}^k) - G^{\star} = F^{\star} - \Lc(\breve{x}^k, \breve{r}^k, \bar{y}^k) \leq  \Lc(x^k, r^k, y^{\star}) - \Lc(\breve{x}^k, \breve{r}^k, \bar{y}^k),
}
since $F^{\star} + G^{\star} = 0$ and $F^{\star} \leq \Lc(x^k, r^k, y^{\star})$. 
Combining this inequality and \eqref{eq:thm1_proof0_telescope_b} yields
\myeqfn{
	0 \leq G(\bar{y}^k) - G^{\star} \leq \frac1{2k} \left[\frac{\rho_0 \norms{K}^2 \norms{x^0 - \breve{x}^k}^2}\gamma + \frac{\norms{y^0 - y^\star}^2}{(1 - \gamma)\rho_0}\right].
}
If $f^*$ is $M_{f^{\ast}}$-Lipschitz continuous, then $\norms{\breve{x}^k - x^0} \leq \sup_x\set{ \norms{x^0 - x} \mid \norms{x} \leq M_{f^{\ast}}} =: D_{f^{\ast}}$. 
Substituting this estimate into the last inequality, we prove \eqref{eq:dual_bound1}.
\proofbox

\beforesubsec
\subsection{\bf The proof of Theorem~\ref{th:faster_convergence_rate}: $\BigO{1/k}$ and $\underline{o}\big(1/(k\sqrt{\log{k}})\big)$ convergence rates when $c > 1$}\label{apdx:th:faster_convergence_rate}
\aftersubsec
Let us first abbreviate $a^2_k := \frac{\rho_0}2 \norms{Kx^k - r^k}^2$, $b_k^2 := \frac{\rho_0\norms{K}^2}{2\gamma}\norms{\tilde{x}^{k} - x^{\star}}^2 + \frac{1}{2(1-\gamma)\rho_0}\norms{\tilde{y}^k - y^\star}^2$, and $\widetilde{\Gc}_k := \Lc(x^{k}, r^k, y^{\star}) - \Lc(x^{\star}, r^{\star}, \bar{y}^{k})$.
Since $(x^{\star}, r^{\star}, y^{\star})$ is a saddle-point of $\Lc$, we have  $\widetilde{\Gc}_k \equiv \Lc(x^k, r^k, y^{\star}) - F^{\star} \geq 0$. 
Using the update rules of parameters at Step~\ref{step:i1} of Algorithm~\ref{alg:A1}, we can derive from \eqref{eq:key_bound1} of Lemma \ref{le:key_bound1} that
\myeqfn{
	\widetilde{\Gc}_{k+1} + \frac{k + c}c a^2_{k+1} \leq \frac{k}{k+c}\Big(\widetilde{\Gc}_k + \frac{k+c-1}c a^2_k \Big) + \frac{c}{k+c}(b_k^2 - b_{k+1}^2) - \frac{(c-1)k}{c(k+c)}a^2_k.
}
Rearranging this estimate, we obtain
\myeqf{eq:th2_proof1}{
\arraycolsep=0.2em
{\!\!\!\!}\begin{array}{rl}
 	(c - 1)\!& \Big(\widetilde{\Gc}_k + \frac{k + c - 1}c a^2_k \Big) \leq (c-1)\Big(\widetilde{\Gc}_k + \frac{2k + c - 1}c a^2_k \Big) \vspace{1ex}\\
	\leq & \Big[(k+c-1)\widetilde{\Gc}_k + \frac{(k+c-1)^2}c a^2_k + c b^2_k \Big] - \left[ (k+c)\widetilde{\Gc}_{k+1} + \frac{(k+c)^2}c a^2_{k+1} + c b^2_{k+1} \right].
\end{array}
\hspace{-3ex}
}
Clearly, the estimate \eqref{eq:th2_proof1} implies
\myeqfn{
	(k + c)\widetilde{\Gc}_{k + 1} + \frac{{(k + c)}^2}c a_{k + 1}^2 + cb_{k + 1}^2 \leq (k + c - 1)\widetilde{\Gc}_k + \frac{{(k + c - 1)}^2}c a_k^2 + cb_k^2.
}
By induction and the definition of $\widetilde{\Gc}_k$, we can easily show from the last estimate that
\myeqf{eq:th2_proof_melodyadd1}{
\arraycolsep=0.3em
\begin{array}{lcl}
	\widetilde{\Lc}(x^k, y^\star) - F^\star & \leq & \Lc(x^k, r^k, y^\star) - F^\star = \widetilde{\Gc}_k \vspace{1ex}\\
	& \leq & \frac{1}{k + c - 1} \left[(c - 1)\widetilde{\Gc}_0 + \frac{{(c - 1)}^2}c a_0^2 + cb_0^2\right] = \frac{R_0^2}{k + c - 1}.
\end{array}
}
By the definition of $R_0^2$ in the statement of Theorem \ref{th:faster_convergence_rate}, and the fact that $Kx^0 - r^0 = 0$ from the initialization step of Algorithm \ref{alg:A1}, we have proved the first assertion of \eqref{eq:primal_dual_bound1_new}. 

Summing up \eqref{eq:th2_proof1} from $i := 0$ to $i := k$, we get
\myeqf{eq:key_bound1_pd2d}{
\arraycolsep=0.3em
{\!\!\!\!}\begin{array}{lcl}
	(c-1)\sum_{i=0}^k\left[\widetilde{\Gc}_i + \frac{(i + c - 1)}c a^2_i\right] & \leq & \left[(c-1)\widetilde{\Gc}_0 + \frac{(c-1)^2}2 a^2_0 + cb_0^2\right]\vspace{1ex}\\
	& & - \left[ (k+c)\widetilde{\Gc}_{k+1} + \frac{(k+c)^2}c a^2_{k+1} + cb^2_{k+1} \right] \vspace{1ex}\\
	& \leq & (c-1)\widetilde{\Gc}_0 + \frac{(c-1)^2}c a^2_0 + cb_0^2 < +\infty.
\end{array}{\!\!\!\!}
}
Since $c - 1 > 0$, $i \leq i + c-1$, and $\widetilde{\Gc}_i \geq 0$, applying Lemma \ref{lm:facts}(c) to \eqref{eq:key_bound1_pd2d}, we get
\myeqf{eq:th2_proof_liminf}{
	\liminf_{k \to \infty} \, k\log{(k)}\left[ \widetilde{\Gc}_k + \frac{ka_k^2}{c}\right] = 0.
}	
In particular, since $0 \leq \widetilde{\Lc}(x^k, y^\star) - F^\star \leq \Lc(x^k, r^k, y^\star) - F^\star = \widetilde{\Gc}_k$, we have proved $\liminf_{k \to \infty} \, k\log{(k)}\big[\widetilde{\Lc}(x^k, y^\star) - F^\star\big] = 0$, which is the second assertion of \eqref{eq:primal_dual_bound1_new}.

Analogous to \eqref{eq:th2_proof_melodyadd1}, we can show that
\myeqf{eq:thm_proof_Kx-r}{
	\norms{Kx^k - r^k} \leq \frac{\sqrt{2c/\rho_0} R_0}{k + c - 1}.
}
By the $M_g$-Lipschitz continuity of $g$, similar to \eqref{eq:th31_proof2}, we can show that
\myeqf{eq:th2_proof5}{
	0 \leq F(x^k) - F^{\star} \leq \Lc(x^k, r^k, y^{\star}) - F^{\star} + (\norms{y^{\star}} + M_g)\norms{Kx^k - r^k}.
}
Combining \eqref{eq:th2_proof_melodyadd1}, \eqref{eq:thm_proof_Kx-r}, and \eqref{eq:th2_proof5}, we get the first assertion of \eqref{eq:primal_bound1_new2}. 
Moreover, applying Lemma~\ref{lm:facts}(d, part (i)) with $u_k := \widetilde{\Gc}_k \geq 0$, $v_k := \norms{Kx^k - r^k}$, $\alpha_1 := \frac{\rho_0}{2c}$, and $\alpha_2 := \norms{y^{\star}} + M_g$ to \eqref{eq:th2_proof_liminf}, we can show that 
\myeqf{eq:th2_proof_liminf2}{
	\liminf_{k\to\infty} \, k\sqrt{\log{k}}\left[\left(\Lc(x^k, r^k, y^\star) - F^\star\right) + (\norms{y^{\star}} + M_g)\norms{Kx^k - r^k}\right] = 0.
}
Furthermore, using the limit \eqref{eq:th2_proof_liminf2} in \eqref{eq:th2_proof5}, we obtain the second assertion of \eqref{eq:primal_bound1_new2}.
\proofbox

\beforesubsec
\subsection{\bf The proof of Corollary~\ref{th:constr_alg_convergence}: Constrained problems}\label{apdx:th:constr_alg_convergence}
\aftersubsec
From \eqref{eq:constr_alg1}, we can write down the optimality condition of $x^{k + 1}$ as
\myeqf{eq:cor_constr_1}{
	0 \in \partial f(x^{k + 1}) + K^\top y^{k + 1} + \nabla\psi(\hat{x}^k) + \tfrac{1}{\beta_k} (x^{k + 1} - \hat{x}^k).
}
By convexity of $f$ and $L_\psi$-smoothness of $\psi$, for any $x \in \R^p$, we have
\myeqf{eq:cor_constr_1b}{
	\psi(x^{k + 1}) \leq \psi(x) + \iprods{\nabla\psi(\hat{x}^k), x^{k + 1} - x} + \frac{L_\psi}2 \norms{x^{k + 1} - \hat{x}^k}^2.
}
Combining \eqref{eq:lm11_est1b}, \eqref{eq:lm11_melodyref1}, \eqref{eq:cor_constr_1}, and \eqref{eq:cor_constr_1b} with $r = r^{k + 1} = b$, for any $x\in\R^p$, we can derive
\myeqfn{
\arraycolsep=0.3em
{\!\!\!\!\!\!}\begin{array}{lcl}
	\Lc_{\rho_k}(x^{k+1}, \tilde{y}^k) & = & f(x^{k + 1}) + \psi(x^{k + 1}) + \iprods{\tilde{y}^k, Kx^{k + 1} - b} + \frac{\rho_k}2 \norms{Kx^{k + 1} - b}^2\vspace{1.25ex}\\
	& \leq & \Lc_{\rho_k}(x, \tilde{y}^k) + \frac{1}{\beta_k}\iprods{ x^{k+1} - \hat{x}^k,  x - \hat{x}^k} - \frac{1}{\beta_k}\norms{x^{k+1} - \hat{x}^k}^2 \vspace{1.25ex}\\
	& & + {~} \frac{\rho_k}{2}\norms{K(x^{k+1} {\!\!} - \hat{x}^k)}^2  - \frac{\rho_k}{2}\norms{K(x - \hat{x}^k)}^2 + \frac{L_\psi}2 \norms{x^{k + 1} - \hat{x}^k}^2.
\end{array}{\!\!\!\!\!\!}
}
Analogous to the proof for Lemma \ref{le:key_bound1} but using the last estimate, we can show that
\myeqf{eq:th3_proof_est4}{
\arraycolsep=0.3em
\hspace{-4ex}\begin{array}{rl}
	\Lc_{\rho_k}{\!\!\!\!} & (x^{k+1}, y) - \Lc(x, \bar{y}^{k + 1}) \leq (1-\tau_k)\left[ \Lc_{\rho_{k-1}}(x^k, y) - \Lc(x, \bar{y}^k)\right]\vspace{1ex}\\
	& + {~} \tfrac{\tau_k^2}{2\beta_k}\left[ \norms{\tilde{x}^k - x}^2 - \norms{\tilde{x}^{k+1} - x}^2\right] + \tfrac{1}{2\eta_k} \left[\norms{\tilde{y}^k - y}^2 - \norms{\tilde{y}^{k+1} - y}^2\right]\vspace{1ex}\\
	& - {~} \tfrac{1}{2}\left(\tfrac{1}{\beta_k} - L_\psi - \tfrac{\rho_k^2\norm{K}^2}{\rho_k - \eta_k}\right)\norms{x^{k+1} {\!\!} - \hat{x}^k}^2 - \frac{(1-\tau_k)}{2}\left[\rho_{k - 1} - (1 - \tau_k)\rho_k\right] \norms{Kx^k - b}^2.
\end{array}
\hspace{-4ex}
}
Using the update \eqref{eq:update_rule1} with $c := 1$ and $\beta_k := \gamma/(\norms{K}^2 \rho_k + \gamma L_\psi)$, for any $x \in \R^p$ and $y \in \R^n$, we follow the same lines as in the proof of Theorem \ref{th:convergence_guarantee1} to derive
\myeqf{eq:cor_constr_proof_L}{
	{\!\!\!\!}\Lc_{\rho_{k-1}}(x^k, y) -\Lc(x, \bar{y}^k)\!\leq\!\frac{1}{2k}\!\left[ \frac{(\rho_0\norms{K}^2+\gamma L_{\psi})}{\gamma}\norms{x^0 - x}^2 + \frac{\norms{y^0 - y}^2}{\rho_0(1-\gamma)} \right],{\!\!\!\!}
}
which implies
\myeqfn{
	F(x^k) + \iprods{y, Kx^k - b} + \frac{\rho_{k-1}}{2}\norms{Kx^k - b}^2 - F^{\star} \leq \frac{R_0^2 (y)}{2k},~~~\forall y\in\R^n,
}
where $R_0^2 (y) := \frac{(\rho_0\norms{K}^2 + \gamma L_{\psi})}{\gamma}\norms{x^0 - x^{\star}}^2 + \frac{1}{\rho_0(1-\gamma)}\norms{y^0 - y}^2$. 
For any $\lambda > 0$, the last inequality leads to
\myeqfn{
	F(x^k) - F^{\star} + \lambda\norms{Kx^k - b} + \frac{\rho_{k - 1}}{2}\norms{Kx^k - b}^2 \leq \frac{1}{2k} \sup\set{R_0^2 (y) \mid \norms{y} \leq \lambda} =: \frac{R_0^2}{2k}.
}
On the other hand, we have $F(x^k) - F^\star \geq  - \iprods{y^{\star}, Kx^k - b} \geq -\norm{y^{\star}}\norms{Kx^k - b}$. 
Combining these expressions, we obtain 
\myeqfn{
\arraycolsep=0.2em
\left\{\begin{array}{lcl}
	\big(\lambda - \norms{y^{\star}}\big)\norms{Kx^k - b} +  \frac{\rho_{k - 1}}{2}\norms{Kx^k - b}^2 &\leq & \frac{R_0^2}{2k},\vspace{1.5ex}\\
	-\norm{y^{\star}}\norms{Kx^k - b} \leq F(x^k) - F^{\star} &\leq & \frac{R_0^2}{2k}.
\end{array}\right.
}
Choosing $\lambda := 2\norm{y^{\star}} + 1$, and noting that $\sup\set{\norms{y^0 - y}^2 \mid \norms{y} \leq \lambda} = \big(\lambda + \norms{y^0} \big)^2 = \big(2\norm{y^{\star}} + \norms{y^0} + 1\big)^2$, we obtain  \eqref{eq:constr_alg_convergence1} from the last expression.

Next, let $\breve{x}^k \in \partial{F^{\ast}}(-K^{\top}\bar{y}^k)$, we have
\myeqfn{
\arraycolsep=0.2em
\begin{array}{lcl}
	G(\bar{y}^k) - G^{\star} &= & \sup_x\set{\iprods{-K^{\top}\bar{y}^k, x} - f(x) - \psi(x)} + \iprods{b, \bar{y}^k} + F^{\star} \vspace{1.5ex}\\
	& \leq & \Lc(x^k, y^{\star}) - f(\breve{x}^k) - \psi(\breve{x}^k) - \iprods{K\breve{x}^k - b, \bar{y}^k} = \Lc(x^k, y^{\star}) - \Lc(\breve{x}^k, \bar{y}^k).
\end{array}
}
Since $\dom{F}$ is bounded, we have $\norms{\breve{x}^k - x^0}  \leq \sup\set{\norms{x - x^0} \mid x\in\dom{F}} =: \Dc_{F}$.
Plugging these two last inequalities in \eqref{eq:cor_constr_proof_L}, we finally obtain \eqref{eq:constr_alg_convergence1_dual}.

For $c > 1$, by the same proof as of \eqref{eq:th2_proof_liminf2} but with $\alpha_2 := \norms{y^{\star}} + 1$, we get 
\myeqfn{
	\liminf_{k\to\infty}\, k\sqrt{\log{k}}\big[ \widetilde{\Gc}_k + (\norms{y^{\star}} + 1)\norms{Kx^k - b}\big] = 0.
}
In addition, from the proof of Theorem~\ref{th:faster_convergence_rate}, we have $F(x^k) - F^{\star} + \iprods{y^{\star}, Kx^k - b} = \Lc(x^k, y^{\star}) - F^{\star} =:  \widetilde{\Gc}_k \geq 0$.
Moreover, $F(x^k) - F^{\star} \geq -\norms{y^{\star}}\norms{Kx^k - b}$.
Combining these inequalities, we can  show that $\vert F(x^k) - F^{\star}\vert \leq \widetilde{\Gc}_k + \norms{y^{\star}}\norms{Kx^k - b}$.
Consequently, we obtain \eqref{eq:constr_alg_convergence_faster1} by combining this inequality and the last limit.
\proofbox

\beforesec
\section{Technical proofs in Section~\ref{sec:strong_convexity}: Strongly convex case}\label{apdx:proofs2}
\aftersec
This appendix provides the full proof of technical results in Section~\ref{sec:strong_convexity}.

\beforesubsec
\subsection{\bf The proof of Lemma~\ref{le:key_est_scvx2}: One-iteration analysis}\label{apdx:le:key_est_scvx2}
\aftersubsec
First, we write down the optimality conditions of the updates of $r^{k + 1}$ and $\tilde{x}^{k + 1}$ in \eqref{eq:pd_scheme3} as
\myeqf{eq:lm_semistr_opt_cond}{
\arraycolsep=0.3em
\left\{\begin{array}{lcl}
	0 & \in & \partial g(r^{k + 1}) + \nabla_r \phi_{\rho_k} (\hat{x}^k, r^{k + 1}, \tilde{y}^k),\vspace{1.5ex}\\
	0 & \in & \partial f(\tilde{x}^{k + 1}) + \nabla_x \phi_{\rho_k} (\hat{x}^k, r^{k + 1}, \tilde{y}^k) + \frac{\tau_k}{\beta_k} (\tilde{x}^{k + 1} - \tilde{x}^k).
\end{array}\right.
}
Let us denote $\breve{x}^{k+1} := (1-\tau_k)x^k + \tau_k\tilde{x}^{k+1}$. 
Then, by convexity of $g$ and strong convexity of $f$ with a strong convexity parameter $\mu_f > 0$, we can derive
\myeqf{eq:lm_semistr_cvx_funcs}{
\arraycolsep=0.3em
\left\{\begin{array}{lcl}
	g(r^{k + 1}) & \leq & (1 - \tau_k)g(r^k) + \tau_k g(r) + \iprods{\nabla g(r^{k + 1}), r^{k + 1} - (1 - \tau_k)r^k - \tau_k r},\vspace{1.5ex}\\
	f(\breve{x}^{k + 1}) & \leq & (1 - \tau_k)f(x^k) + \tau_k f(x) + \tau_k \iprods{\nabla f(\tilde{x}^{k + 1}), \tilde{x}^{k + 1} - x}\vspace{1ex}\\
	& & - {~} \frac{\tau_k \mu_f}2 \norms{\tilde{x}^{k + 1} - x}^2 - \frac{\tau_k (1 - \tau_k)\mu_f}2 \norms{\tilde{x}^{k + 1} - x^k}^2,
\end{array}\right.
}
where $\nabla g(r^{k + 1}) \in \partial g(r^{k + 1})$ and $\nabla f(x^{k + 1}) \in \partial f(x^{k + 1})$ are subgradients. 

\noindent Next, using Lemma \ref{lm:facts}(a) three times by setting $(x', r')$ as $(x^k, r^k)$, $(x^{k + 1}, r^{k + 1})$, and $(x,r)$, and $(x,r)$ as $(\hat{x}^k, r^{k+1})$, respectively, similar to \eqref{eq:lml1_melodyref1_prep}, we can eventually derive 
\myeqf{eq:lm_semistr_phi}{
\hspace{-4ex}\arraycolsep=0.05em
\begin{array}{rl}
	\phi_{\rho_k} &{\!\!}(x^{k + 1}, r^{k + 1}, \tilde{y}^k) = (1 - \tau_k)\phi_{\rho_k} (x^k, r^k, \tilde{y}^k) + \tau_k \phi_{\rho_k} (x, r, \tilde{y}^k)\vspace{1.2ex}\\
	&{~} + {~} \iprods{\nabla_x \phi_{\rho_k} (\hat{x}^k, r^{k + 1}, \tilde{y}^k),\ x^{k + 1} - (1 - \tau_k)x^k - \tau_k x}\vspace{1.2ex}\\
	&{~} + {~} \iprods{\nabla_r \phi_{\rho_k} (\hat{x}^k, r^{k + 1}, \tilde{y}^k),\ r^{k + 1} - (1 - \tau_k)r^k - \tau_k r}   + \frac{\rho_k}2 \norms{K(x^{k + 1} - \hat{x}^k)}^2 \vspace{1.2ex}\\
	&{~} \Big\{ - {~} \frac{(1 - \tau_k)\rho_k}2 \norms{K\hat{x}^k - r^{k + 1} {\!} - (Kx^k \!-\! r^k)}^2 - \!\frac{\tau_k \rho_k}2 \norms{K\hat{x}^k - r^{k + 1} {\!\!} - (Kx - r)}^2\Big\}_{[\Tc_1]},
\end{array}
\hspace{-4ex}
}
where we define the last line as $\Tc_1$. 

\noindent  Combining \eqref{eq:lm_semistr_opt_cond}, \eqref{eq:lm_semistr_cvx_funcs}, and \eqref{eq:lm_semistr_phi}, we get
\myeqf{eq:proof_semistr_L1}{
\hspace{-2ex}
\arraycolsep=0.3em
\begin{array}{rl}
	\Lc_{\rho_k} (x^{k + 1}, & r^{k + 1}, \tilde{y}^k) = f(x^{k + 1}) + g(r^{k + 1}) + \phi_{\rho_k} (x^{k + 1}, r^{k + 1}, \tilde{y}^k) \vspace{1ex}\\
	\stackrel{\eqref{eq:lm_semistr_opt_cond}-\eqref{eq:lm_semistr_phi}}\leq & (1 - \tau_k)\left[g(r^k) + \phi_{\rho_k} (x^k, r^k, \tilde{y}^k)\right] + \tau_k \left[g(r) + \phi_{\rho_k} (x, r, \tilde{y}^k)\right] + \Tc_1\vspace{2ex}\\
	& \left.\begin{array}{l}
		+ {~} \frac{\rho_k}2 \norms{K(x^{k + 1} - \hat{x}^k)}^2 + \iprods{\nabla_x \phi_{\rho_k} (\hat{x}^k, r^{k + 1}, \tilde{y}^k),\ x^{k + 1} - \hat{x}^k}\vspace{1.25ex}\\
		+ {~} f(x^{k + 1}) + \iprods{\nabla_x \phi_{\rho_k} (\hat{x}^k, r^{k + 1}, \tilde{y}^k), \hat{x}^k - (1 - \tau_k) x^k - \tau_k x}
\end{array}\right\}\hfill =: \Tc_2.
\end{array}
\hspace{-2ex}
}
To estimate $\Tc_2$, notice that by the optimality condition of the $x^{k + 1}$-update in \eqref{eq:pd_scheme3} and the $\mu_f$-strong convexity of $f$, we can show that
\myeqf{eq:lm_semistr_f_opt}{
\arraycolsep=0.3em
\begin{array}{rl}
	f(x^{k + 1}) & + {~} \frac{\rho_k \norms{K}^2}2 \norms{x^{k + 1} - \hat{x}^k}^2 + \iprods{\nabla_x \phi_{\rho_k} (\hat{x}^k, r^{k + 1}, \tilde{y}^k), x^{k + 1} - \hat{x}^k}\vspace{1ex}\\
	\leq & f(\breve{x}^{k + 1}) + \frac{\rho_k \norms{K}^2}2 \norms{\breve{x}^{k + 1} - \hat{x}^k}^2 + \iprods{\nabla_x \phi_{\rho_k} (\hat{x}^k, r^{k + 1}, \tilde{y}^k), \breve{x}^{k + 1} - \hat{x}^k}\vspace{1ex}\\
	& - {~} \frac{\rho_k \norms{K}^2 + \mu_f}2 \norms{\breve{x}^{k + 1} - x^{k + 1}}^2.
\end{array}
}
Using the above inequality as well as \eqref{eq:lm_semistr_opt_cond} and \eqref{eq:lm_semistr_cvx_funcs}, we can upper bound
\myeqf{eq:proof_semistr_T2}{
\arraycolsep=0.3em
\begin{array}{lcl}
	\Tc_2 {\!\!\!\!}& \stackrel{\eqref{eq:lm_semistr_f_opt}}\leq & f(\breve{x}^{k + 1}) + \frac{\rho_k \norms{K}^2}2 \norms{\breve{x}^{k + 1} - \hat{x}^k}^2 - \frac{\rho_k \norms{K}^2 + \mu_f}2 \norms{\breve{x}^{k + 1} - x^{k + 1}}^2\vspace{1ex}\\
	& & + {~} \iprods{\nabla_x \phi_{\rho_k} (\hat{x}^k, r^{k + 1}, \tilde{y}^k), \breve{x}^{k + 1} - (1 - \tau_k) x^k - \tau_k x}  \vspace{1ex}\\
	& \stackrel{\eqref{eq:lm_semistr_opt_cond}-\eqref{eq:lm_semistr_cvx_funcs}}\leq & (1 - \tau_k)f(x^k) + \tau_k f(x) + \frac{\tau_k^2}{\beta_k} \iprods{\tilde{x}^{k + 1} - \tilde{x}^k, x - \tilde{x}^{k + 1}}\vspace{1ex}\\
	& & + {~} \iprods{\nabla_x \phi_{\rho_k} (\hat{x}^k, r^{k + 1}, \tilde{y}^k), \breve{x}^{k + 1} - (1 - \tau_k)x^k - \tau_k \tilde{x}^{k + 1}} \vspace{1ex}\\
	& & - {~} \frac{\rho_k \norms{K}^2}2 \norms{\breve{x}^{k + 1} - \hat{x}^k}^2 - \frac{(\rho_k \norms{K}^2 + \mu_f)}{2} \norms{\breve{x}^{k + 1} - x^{k + 1}}^2 \vspace{1ex}\\
	& & - {~} \frac{\tau_k \mu_f}2 \norms{\tilde{x}^{k + 1} - x}^2 - \frac{\tau_k (1 - \tau_k)\mu_f}2 \norms{\tilde{x}^{k + 1} - x^k}^2\vspace{1ex}\\
	& = & (1 - \tau_k)f(x^k) + \tau_k f(x) + \frac{\tau_k^2}{2\beta_k} \norms{\tilde{x}^k - x}^2 - \frac{\tau_k (\tau_k + \beta_k \mu_f)}{2\beta_k} \norms{\tilde{x}^{k + 1} - x}^2\vspace{1ex}\\
	& & - {~} \frac{(1- \rho_k \beta_k \norms{K}^2)}{2\beta_k} \norms{\breve{x}^{k\!+\!1}\!-\!\hat{x}^k}^2\!-\!\frac{(\rho_k \norms{K}^2 + \mu_f)}{2}\norms{\breve{x}^{k\!+\!1}\!-\!x^{k\!+\!1}}^2 \vspace{1ex}\\
	&& - {~} \frac{\tau_k (1- \tau_k)\mu_f}2 \norms{\tilde{x}^{k\!+\!1}\!-\!x^k}^2,
	\end{array}
}
where we have used $\breve{x}^{k + 1} - \hat{x}^k = \tau_k (\tilde{x}^{k + 1} - \tilde{x}^k)$. 
Substituting \eqref{eq:proof_semistr_T2} into \eqref{eq:proof_semistr_L1}, we have
\myeqf{eq:proof_semistr_L2}{
\arraycolsep=0.3em
\begin{array}{rl}
	\Lc_{\rho_k} {\!\!\!\!}& (x^{k + 1}, r^{k + 1}, \tilde{y}^k) \leq (1 - \tau_k)\Lc_{\rho_k} (x^k, r^k, \tilde{y}^k) + \tau_k \Lc_{\rho_k} (x, r, \tilde{y}^k) + \Tc_1\vspace{1ex}\\
	& + {~} \frac{\tau_k^2}{2\beta_k} \norms{\tilde{x}^k - x}^2 - \frac{\tau_k (\tau_k + \beta_k \mu_f)}{2\beta_k} \norms{\tilde{x}^{k + 1} - x}^2 - \frac{\tau_k (1 - \tau_k)\mu_f}{2\beta_k} \norms{\tilde{x}^{k + 1} - x^k}^2\vspace{1ex}\\
	& - {~} \frac{(1 - \rho_k \beta_k\norms{K}^2)}{2\beta_k} \norms{\breve{x}^{k + 1} - \hat{x}^k}^2 - \frac{(\rho_k \norms{K}^2 + \mu_f)}{2} \norms{\breve{x}^{k + 1} - x^{k + 1}}^2.
\end{array}
}
By the definition of $\tilde{y}^{k + 1}$ from \eqref{eq:pd_scheme3} and that of $\bar{y}^{k + 1}$, the equations \eqref{eq:lm11_est7}, \eqref{eq:Lemma2_Tc1}, and \eqref{eq:T2_proof5} still hold. 
Substituting them into \eqref{eq:proof_semistr_L2}, and using the expression of $\Tc_1$, we get
\myeqf{eq:proof_semistr_L3}{
\arraycolsep=0.2em
{\!\!\!\!}\begin{array}{rl}
	\Lc_{\rho_k}(x^{k + 1}, &r^{k + 1}, y) - \Lc(x, r, \bar{y}^{k + 1}) \leq (1 - \tau_k)\left[\Lc_{\rho_{k - 1}}(x^k, r^k, y) - \Lc(x, r, \bar{y}^k)\right]\vspace{1.25ex}\\
	& \begin{rcases}
	+ {~}  \frac{(1 - \tau_k)}{2}\left(\rho_k - \rho_{k - 1}\right) \norms{Kx^k - r^k}^2 - \frac{\tau_k \rho_k}2 \norms{K\hat{x}^k - r^{k + 1}}^2\vspace{1ex}\\
	- {~} \frac{(1 - \tau_k)\rho_k}2 \norms{K\hat{x}^k - r^{k + 1} - (Kx^k - r^k)}^2 + \frac1{2\eta_k} \norms{\tilde{y}^{k + 1} - \tilde{y}^k}^2
	\end{rcases}\hfill =: \Tc_3\vspace{1.25ex}\\
	& - {~} \frac{(1 - \rho_k \beta_k \norms{K}^2)}{2\beta_k} \norms{\breve{x}^{k + 1} - \hat{x}^k}^2 - \frac{(\rho_k \norms{K}^2 + \mu_f)}{2} \norms{\breve{x}^{k + 1} - x^{k + 1}}^2\hfill =: \Tc_4\vspace{1.25ex}\\
	& + {~} \frac{\tau_k^2}{2\beta_k} \norms{\tilde{x}^k - x}^2 - \frac{\tau_k (\tau_k + \beta_k \mu_f)}{2\beta_k} \norms{\tilde{x}^{k + 1} - x}^2 - \frac{\tau_k (1 - \tau_k)\mu_f}{2\beta_k} \norms{\tilde{x}^{k + 1} - x^k}^2 \vspace{1ex}\\
	&  + {~}  \frac1{2\eta_k} \left[ \norms{\tilde{y}^k - y}^2 - \norms{\tilde{y}^{k + 1} - y}^2\right].
\end{array}{\!\!}
}
Since $\rho_k > \eta_k$, using the same lines as \eqref{eq:lm11_est8}-\eqref{eq:lm11_est9} in the proof of Lemma \ref{le:key_bound1}, we have
\myeqf{eq:proof_semistr_T4}{
	\Tc_3 \leq \frac{\rho_k \eta_k}{2(\rho_k - \eta_k)} \norms{K(x^{k + 1} - \hat{x}^k)}^2 - \frac{(1 - \tau_k)}{2}\left[\rho_{k - 1} - (1 - \tau_k)\rho_k\right] \norms{Kx^k - r^k}^2.
}
Applying Lemma \ref{lm:facts}(b) on $\Tc_4$ with $\alpha_1 := \frac{1- \rho_k \beta_k \norms{K}^2}{2\beta_k}$ and $\alpha_2 := \frac{\rho_k \norms{K}^2 + \mu_f}2$, and noting that $\rho_k\beta_k\norms{K}^2 < 1$, we can further show that
\myeqf{eq:proof_semistr_T5}{
	\Tc_4 \leq -\frac{\alpha_1 \alpha_2}{\alpha_1 + \alpha_2} \norms{x^{k + 1} - \hat{x}^k}^2 \leq - \frac{\rho_k}2 (1 - \rho_k \beta_k \norms{K}^2)\norms{K(x^{k + 1} - \hat{x}^k)}^2.
}
Substituting \eqref{eq:proof_semistr_T4} and \eqref{eq:proof_semistr_T5} into \eqref{eq:proof_semistr_L3}, we finally arrive at \eqref{eq:key_est_scvx2}.
\proofbox

\beforesubsec
\subsection{\bf The proof of Theorem~\ref{th:convergence_guarantee_scvx1}: $\BigO{1/k^2}$ convergence rates}\label{apdx:th:big_O_convergence_rates_scvx}
\aftersubsec
By the parameter update rule \eqref{eq:update_rule_scvx1}, we can easily verify that
\myeqfn{
\left\{\begin{array}{l}
	\frac{\tau_k^2}{2\beta_k} \leq \frac{(1 - \tau_k)\tau_{k - 1} (\tau_{k - 1} + \beta_{k - 1} \mu_f)}{2\beta_{k - 1}}, \qquad \frac1{2\eta_k} = \frac{1 - \tau_k}{2\eta_{k - 1}},\vspace{1.5ex}\\
	1 - \rho_k \beta_k \norms{K}^2 - \frac{\eta_k}{\rho_k - \eta_k} = \frac{(1 - \tau_k)}{2}\left[\rho_{k - 1} - (1 - \tau_k)\rho_k\right] = 0.
\end{array}\right.
}
Applying these conditions to Lemma \ref{le:key_est_scvx2}, we can simplify \eqref{eq:key_est_scvx2} as
\myeqfn{
\arraycolsep=0.3em
\begin{array}{rl}
	\Lc_{\rho_k} (x^{k + 1}, r^{k + 1}, y)\ - & \Lc(x, r, \bar{y}^{k + 1}) + \frac{\tau_k (\tau_k + \beta_k \mu_f)}{2\beta_k} \norms{\tilde{x}^{k + 1} - x}^2 + \frac1{2\eta_k} \norms{\tilde{y}^{k + 1} - y}^2\vspace{1ex}\\
	\leq & (1 - \tau_k) \Big[\Lc_{\rho_{k - 1}} (x^k, r^k, y) - \Lc(x, r, \bar{y}^k) \vspace{1ex}\\
	&{~~~~~~~~~~~~} + {~} \frac{\tau_{k - 1} (\tau_{k - 1} + \beta_{k - 1} \mu_f)}{2\beta_{k - 1}} \norms{\tilde{x}^k - x}^2 + \frac1{2\eta_{k - 1}} \norms{\tilde{y}^k - y}^2 \Big].
\end{array}
}
By induction, the above inequality implies
\myeqfn{
\hspace{-2ex}\arraycolsep=0.3em
\begin{array}{lcl}
	\Lc_{\rho_{k - 1}} (x^k, r^k, y) - \Lc(x, r, \bar{y}^k) & \leq & \left[\prod\limits_{i = 1}^{k - 1} (1 - \tau_i)\right] \Big[(1 - \tau_0)\left(\Lc_{\rho_{-1}} (x^0, r^0, y) - \Lc(x, r, \bar{y}^0)\right) \vspace{1ex}\\
	& & + {~} \frac{\tau_0^2}{2\beta_0} \norms{\tilde{x}^0 - x}^2 + \frac1{2\eta_0} \norms{\tilde{y}^0 - y}^2 \Big]\vspace{1ex}\\
	& = & \tau_{k - 1}^2 \left[ \frac{\tau_0^2}{2\beta_0} \norms{\tilde{x}^0 - x}^2 + \frac1{2\eta_0} \norms{\tilde{y}^0 - y}^2\right]\vspace{1ex}\\
	& \leq & \frac4{{(k + 1)}^2} \Big[ \frac{\rho_0 \norms{K}^2\norms{x^0 - x}^2}{2\Gamma}  + \frac{\norms{y^0 - y}^2}{2(1 - \gamma)\rho_0}\Big],
\end{array}
\hspace{-2ex}
}
where we have used $1 - \tau_k = \tau_k^2/\tau_{k - 1}^2$, $\tau_k \leq 2/(k + 2)$, and the parameter initialization \eqref{eq:update_rule_scvx_init}. 
The remaining conclusions of Theorem \ref{th:convergence_guarantee_scvx1} follow the same lines as the proof of Theorem \ref{th:convergence_guarantee1}.
Thus we omit the details here.
\proofbox

\beforesubsec
\subsection{\bf The proof of Theorem~\ref{th:small_o_convergence_rate_scvx}: $\BigO{1/k^2}$ and $\underline{o}\big(1/(k^2\sqrt{\log{k}})\big)$ convergence rates}\label{apdx:th:small_o_convergence_rate_scvx}
\aftersubsec
Similar to the proof of Theorem~\ref{th:faster_convergence_rate}, we abbreviate $a_k^2 := \frac{\rho_0}2 \norms{Kx^k - r^k}^2,\ b_k^2 := \frac{\rho_0 \norms{K}^2}{2\Gamma} \norms{\tilde{x}^k - x^\star}^2$, $d_k^2 := \frac1{2(1 - \gamma)\rho_0} \norms{\tilde{y}^k - y^\star}^2$, and $\widetilde{\Gc}_k := \Lc(x^k, r^k, y^{\star}) - \Lc(x^{\star}, r^{\star}, \bar{y}^{k}) \geq 0$. 
We can rewrite \eqref{eq:key_est_scvx2} in Lemma~\ref{le:key_est_scvx2} as follows:
\myeqfn{
\arraycolsep=0.2em
\begin{array}{lcl}
	\widetilde{\Gc}_{k + 1} + {\left(\frac{k + c}c\right)}^2 a_{k + 1}^2 & \leq & \frac{k}{k + c} \left[\widetilde{\Gc}_k + {\left(\frac{k + c - 1}c\right)}^2 a_k^2\right] + (b_k^2 - b_{k + 1}^2) 
	- \frac{c\Gamma\mu_f}{(k + c)\rho_0 \norms{K}^2} b_{k + 1}^2\vspace{1ex}\\
	& & + {~}  {\left(\frac{c}{k + c}\right)}^2 (d_k^2 - d_{k + 1}^2) - \frac{k}{c^2 (k + c)} \left[{(k + c - 1)}^2 - k(k + c)\right]a_k^2.
\end{array}
}
Multiplying both sides of this estimate by $(k+c)^2$ and rearranging the result, we get
\myeqf{eq:proof_str_o_pre_tele}{
\arraycolsep=0.2em
\begin{array}{lcl}
	R_{k + 1}^2 & := & {(k+c)}^2\widetilde{\Gc}_{k+1} + \frac{{(k+c)}^4}{c^2} a_{k+1}^2 + \left[{(k + c)}^2 + \frac{c(k + c)\Gamma\mu_f}{\rho_0 \norms{K}^2}\right]b_{k+1}^2 + c^2 d_{k+1}^2\vspace{1ex}\\
	& \leq & k(k+c)\widetilde{\Gc}_k + \frac{k^2 {(k+c)}^2}{c^2} a_k^2 + {(k+c)}^2 b_k^2 + c^2 d_k^2.
\end{array}
}
If $c > 2$ and $0 < \rho_0 \leq \frac{c(c - 1)\Gamma\mu_f}{(2c - 1)\norms{K}^2}$, then the above right-hand-side is bounded by $R_k^2$, and we have
\myeqf{eq:proof_str_o_tele}{
\arraycolsep=0.2em
{\!\!\!\!}\begin{array}{rl}
	(c - 2) & \Big[(k + c - 1)\widetilde{\Gc}_k + \frac{{(k + c - 1)}^3}{c^2} a_k^2\Big] \leq  \left[(c - 2)k + {(c - 1)}^2\right]\widetilde{\Gc}_k \vspace{1ex}\\
	& + {~} \frac{1}{c^2} \left[{(k + c - 1)}^4 - k^2 {(k + c)}^2\right]a_k^2\vspace{1ex}\\
	\leq & \left[{(k+c - 1)}^2\widetilde{\Gc}_k + \frac{{(k+c - 1)}^4}{c^2} a_k^2 + \left({(k + c - 1)}^2 + \frac{c(k + c - 1)\Gamma\mu_f}{\rho_0 \norms{K}^2}\right)b_k^2 + c^2 d_k^2\right]\vspace{1ex}\\
	& - \left[{(k+c)}^2\widetilde{\Gc}_{k+1} + \frac{{(k+c)}^4}{c^2} a_{k+1}^2 + \left({(k + c)}^2 + \frac{c(k + c)\Gamma\mu_f}{\rho_0 \norms{K}^2}\right)b_{k+1}^2 + c^2 d_{k+1}^2\right]\vspace{1ex}\\
	= & R_k^2 - R_{k + 1}^2.
\end{array}{\!\!\!\!}
}
By induction and the definitions of $\widetilde{\Gc}_k$ and $R_k^2$, we can show that
\myeqf{eq:th4_proof_melodyadd}{
	\widetilde{\Lc}(x^k, y^\star) - F^\star \leq \Lc(x^k, r^k, y^\star) - F^\star = \widetilde{\Gc}_k   \leq \frac{R_k^2}{{(k + c - 1)}^2} \leq \frac{R_0^2}{{(k + c - 1)}^2}.
}
By the initialization of Algorithm \ref{alg:A2}, we have proved the first assertion of \eqref{eq:primal_bound_scvx1b}. 

Summing up \eqref{eq:proof_str_o_tele} from $i := 0$ to $i := k$, we get
\myeqfn{
	(c - 2)\sum_{i = 0}^k \left[(i + c - 1)\widetilde{G}_i + \tfrac{{(i + c - 1)}^3}{c^2} a_i^2 \right] \leq R_0^2 - R_{k + 1}^2 \leq R_0^2 < +\infty.
}
Since $c - 2 > 0$ and $i \leq i + c - 1$, applying Lemma \ref{lm:facts}(c) to the last expression yields{\!\!\!}
\myeqfn{
	\liminf_{k \to \infty} \,  k^2\log{k} \left[ \widetilde{\Gc}_k + \frac{k^2 a_k^2}{c^2} \right] = 0.
}
Since $0 \leq \widetilde{\Lc}(x^k, y^\star) - F^\star \leq \Lc(x^k, r^k, y^{\star}) - F^{\star} = \widetilde{\Gc}_k$, we can easily obtain the second assertion of \eqref{eq:primal_bound_scvx1b} from this limit.

The remaining statements of Theorem \ref{th:small_o_convergence_rate_scvx} can be proved in a similar manner as of Theorem \ref{th:faster_convergence_rate}, but by applying Lemma~\ref{lm:facts}(d, part (ii)) to prove the limit in \eqref{eq:primal_bound_scvx1b_primal}. Thus we omit the details here.
\proofbox

\beforesubsec
\subsection{\bf The proof of Corollary~\ref{th:constr_alg_convergence_scvx}: Constrained problems}\label{apdx:th:constr_alg_convergence_scvx}
\aftersubsec
The augmented Lagrangian associated with problem \eqref{eq:constr_cvx2} is $\Lc_{\rho} (x, w, y) := f(x) + \psi(w) + \iprods{y, Kx + Bw - b} + \frac\rho2 \norms{Kx + Bw - b}^2$. Let $\tilde{w}^{k + 1} := \frac1{\tau_k} \left[w^{k + 1} - (1 - \tau_k)w^k\right]$. The optimality condition of the $w^{k + 1}$-update in \eqref{eq:pd_scheme3_semistr} and the convexity of $\psi$ imply for $w \in \R^q$ that
\myeqfn{
\arraycolsep=0.2em
\begin{array}{lcl}
	\psi(w^{k + 1}) & \leq & (1 - \tau_k)\psi(w^k) + \tau_k \psi(w)\vspace{1.25ex}\\
	& & + {~} \tau_k \iprods{B^\top [\tilde{y}^k + \rho_k (K\hat{x}^k + Bw^{k + 1} - b)] + \nu_0 (w^{k + 1} - \hat{w}^k), w - \tilde{w}^{k + 1}}.
\end{array}
}
Using this estimate, we follow the same lines as the proof of Lemma \ref{le:key_est_scvx2} to derive
\myeqf{eq:key_lm_semistr}{
\hspace{-4ex}\arraycolsep=0.2em
\begin{array}{rl}
	\Lc_{\rho_k}{\!\!\!}& (x^{k + 1}, w^{k + 1}, y) - \Lc(x, w, \bar{y}^{k + 1}) \leq (1 - \tau_k)\left[\Lc_{\rho_{k - 1}} (x^k, w^k, y) - \Lc(x, w, \bar{y}^k)\right]\vspace{1.25ex}\\
	& + {~} \frac{\tau_k^2}{2\beta_k} \norms{\tilde{x} - x}^2 - \frac{\tau_k (\tau_k + \beta_k \mu_f)}{2\beta_k} \norms{\tilde{x}^{k + 1} - x}^2 
	+ \frac{\tau_k^2 \nu_0}2 \left[ \norms{\tilde{w}^k - w}^2 - \norms{\tilde{w}^{k + 1} - w}^2\right] \vspace{1.25ex}\\
	& + {~} \frac1{2\eta_k} \left[ \norms{\tilde{y}^k - y}^2 - \norms{\tilde{y}^{k + 1} - y}^2\right] - \frac{(1 - \tau_k)}{2}\left[\rho_{k - 1} - (1 - \tau_k)\rho_k\right]\norms{Kx^k + Bw^k - b}^2\vspace{1.25ex}\\
	& - {~} \frac{\rho_k}2 \left(1 - \rho_k \beta_k \norms{K}^2 - \frac{\eta_k}{\rho_k - \eta_k}\right)\norms{K(x^{k + 1} - \hat{x}^k)}^2 - \frac{\nu_k}2 \norms{w^{k + 1} - \hat{w}^k}^2.
\end{array}
\hspace{-4ex}
}
Plugging the parameter updates \eqref{eq:update_rule_scvx1} and \eqref{eq:update_rule_scvx_init} into \eqref{eq:key_lm_semistr}, we can derive \eqref{eq:constr_alg_convergence2} following the same arguments as in the proof of Corollary \ref{th:constr_alg_convergence}.

If we plug the parameter updates \eqref{eq:update_rule_scvx1} and \eqref{eq:small_o_update_rule_scvx_init} into \eqref{eq:key_lm_semistr}, then we can derive \eqref{eq:constr_alg_convergence_faster2} following the same arguments as in the proof of Theorem \ref{th:small_o_convergence_rate_scvx}. 
We omit the details.
\proofbox

\bibliographystyle{siamplain}

\end{document}